\documentclass[11pt]{amsart}
\usepackage[utf8]{inputenc}
\usepackage[margin=1.1in]{geometry}
\usepackage{amsmath,amssymb,amsthm,mathtools}
\usepackage{enumitem}
\usepackage{microtype}
\usepackage{hyperref}
\usepackage{tikz}
\usetikzlibrary{arrows.meta,calc,positioning,decorations.pathreplacing,decorations.markings}
\usetikzlibrary{arrows.meta,calc}

\hypersetup{
 colorlinks=true,
 linkcolor=blue!50!black,
 citecolor=blue!50!black,
 urlcolor=blue!50!black
}

\newtheorem{theorem}{Theorem}[section]
\newtheorem{proposition}[theorem]{Proposition}
\newtheorem{lemma}[theorem]{Lemma}
\newtheorem{corollary}[theorem]{Corollary}
\newtheorem{definition}[theorem]{Definition}
\newtheorem{remark}[theorem]{Remark}
\newtheorem{convention}[theorem]{Convention}
\newtheorem{introtheorem}{Theorem}

\newcommand{\Mod}{\operatorname{Mod}}
\newcommand{\Diff}{\operatorname{Diff}}
\newcommand{\Stab}{\operatorname{Stab}}
\newcommand{\im}{\operatorname{im}}

\title[Rim surgery and extendable mapping classes]{Extendable mapping classes of knotted surfaces obtained by rim surgery in \texorpdfstring{$S^4$}{S4}}
\author{Weizhe Niu}
\address{Yau Mathematical Sciences Center, Tsinghua University}
\email{weizheniu@mail.tsinghua.edu.cn}

\begin{document}

\begin{abstract}
Let \(\Sigma_g^0\subset S^4\), \(g\ge3\), be the standard unknotted closed
oriented surface, and let \(a\subset\Sigma_g^0\) be an oriented nonseparating
curve.  For every nontrivial knot \(J\subset S^3\), let
\(\Sigma_{g,a,J}\subset S^4\) be the surface obtained from \(\Sigma_g^0\) by
ordinary untwisted rim surgery along \(a\).  We compute its extendable
mapping-class subgroup exactly:
\[
E(\Sigma_{g,a,J})
=
\Stab_{\Mod(\Sigma_g)}(q_0)
\cap
\Stab_{\Mod(\Sigma_g)}(\Gamma_\mu(J)\cdot[a]).
\]
Here \(q_0\) is the Rokhlin quadratic form of the standard embedding,
\([a]\in H_1(\Sigma_g;\mathbb Z)\) is the oriented rim homology class, and
\(\Gamma_\mu(J)\subset\{\pm1\}\) records whether a meridian-preserving
diffeomorphism of the knot exterior can preserve or reverse the preferred
longitude.  Thus ordinary rim surgery cuts Hirose's unknotted extendable
subgroup by the stabilizer of the rim homology class, with the only additional
ambiguity coming from this peripheral symmetry of \(J\). We also prove a prescribed-mapping-class classification for such
ambient pairs \((S^4,\Sigma_{g,a,J})\). More precisely, given two such pairs and
\(f\in\Mod(\Sigma_g)\), we characterize when \(f\) is induced by an
orientation-preserving pair diffeomorphism in terms of the Rokhlin quadratic
form, the rim homology classes, and the meridian--longitude
symmetries of the knot exteriors.
\end{abstract}
\maketitle

\section{Introduction}

A natural way to measure the symmetry of a knotted surface \(\Sigma\subset X^4\) is to ask which abstract mapping classes of \(\Sigma\) are induced by diffeomorphisms of the ambient pair. Thus one studies the extendable subgroup \[ E(X,\Sigma)= \operatorname{im}\left( \pi_0\Diff^+(X,\Sigma)\longrightarrow \Mod(\Sigma) \right), \] which records the mapping classes induced by ambient diffeomorphisms. This problem also appears as Problem 4.39 in the K3 problem list
\cite{BaykurKirbyRubermanK3}. Even for surfaces in \(S^4\), this subgroup is sensitive both to the embedding and to the four-dimensional complement. For the standard unknotted orientable surface \(\Sigma_g^0\subset S^4\), Montesinos treated the genus-one case \cite{Montesinos1983}, and Hirose proved in all genera that the extendable subgroup is exactly the stabilizer of the Rokhlin quadratic form \cite{Hirose2002}: \[ E(S^4,\Sigma_g^0)=\Stab_{\Mod(\Sigma_g)}(q_0). \] Hirose also studied extendable mapping classes for certain nontrivial \(T^2\)-knots \cite{Hirose1993}, and for nonorientable standard surfaces the corresponding answer is governed by the Guillou--Marin quadratic form \cite{Hirose2012Nonorientable}. More generally, the problem has been studied from several complementary viewpoints, including flexible knotted surfaces in smooth four-manifolds \cite{HiroseYasuhara2008}, periodic mapping classes over \(S^4\) \cite{WangWang2024}, and recent criteria and constructions for extendible and non-extendible mapping classes of surfaces in four-manifolds \cite{LawandeSaha2025}.

More recently, Q.~Liu constructed knotted surfaces in \(S^4\) in every positive genus for which the action of the extendable self-homeomorphisms on first homology has finite image \cite{Liu2026}. In genus greater than one, this does not determine the full extendable mapping-class subgroup, since mapping classes acting trivially on homology remain uncontrolled. Baykur and Sunukjian used iterated rim surgery and relative Seiberg--Witten invariants to obtain projective homological restrictions on extendable groups; they also constructed a totally geodesic positive-genus surface in a closed hyperbolic \(4\)-manifold whose smooth and topological extendable mapping-class groups are trivial \cite{BaykurSunukjian2026}. Thus substantial homological rigidity and specially constructed examples with trivial full groups are known, but exact computations of proper full extendable subgroups for natural families of knotted surfaces in \(S^4\) remain comparatively rare.

The purpose of this paper is to give such a computation for one of the basic natural ways of producing knotted surfaces in \(S^4\): ordinary untwisted rim surgery on the standard surface. We give an exact formula for the extendable subgroup, together with a corresponding classification of
pairs obtained by one ordinary rim surgery.

More precisely, let \(a\subset\Sigma_g^0\) be an oriented nonseparating curve.
The curve \(a\) determines a product torus
\[
a\times S^1_{\mu_\Sigma}
\subset
\partial(S^4\setminus\nu\Sigma_g^0)
\cong
\Sigma_g^0\times S^1_{\mu_\Sigma},
\]
where \(S^1_{\mu_\Sigma}\) is the oriented meridian circle of the surface.
After pushing this torus slightly into the complement, we obtain the rim torus
\(R_a\).  Let \(J\subset S^3\) be a knot.  Ordinary untwisted rim surgery
removes \(\nu R_a\cong T^2\times D^2\) and glues in \(S^1\times E(J)\), where
\(E(J)=S^3\setminus\nu J\), by the convention
\[
s\mapsto \alpha,
\qquad
\mu_J\mapsto \mu_\Sigma,
\qquad
\lambda_J\mapsto \delta.
\]
Here \(s\) is the \(S^1\)-factor of \(S^1\times E(J)\), \(\mu_J\) and
\(\lambda_J\) are the meridian and preferred longitude of \(J\), \(\alpha\) is
the rim-longitude direction, and \(\delta\) is the rim-torus meridian.  We
denote the resulting surface by \(\Sigma_{g,a,J}\). We prove that
\[
\pi_1(S^4\setminus\nu\Sigma_{g,a,J})\cong \pi_1(E(J))=G_J,
\]
and, under this isomorphism,
\[
\mu_\Sigma\mapsto \mu_J,
\qquad
c\mapsto \lambda_J^{a\cdot c}
\]
for each loop \(c\subset\Sigma_g\) on the surface boundary.  Thus the surface
boundary map detects the functional \(c\mapsto a\cdot c\), and hence the
oriented homology class \([a]\), though not the isotopy class of \(a\).  Since
the standard unknotted surface complement has fundamental group \(\mathbb Z\),
the surface \(\Sigma_{g,a,J}\) is knotted whenever \(J\) is nontrivial.

\begin{definition}[meridian-preserving longitude-sign group]
\label{def:Gamma-mu}
Let \(J\subset S^3\) be a nontrivial oriented knot with oriented meridian
\(\mu_J\) and preferred longitude \(\lambda_J\).  Define
\[
\Gamma_\mu(J)
=
\left\{
\varepsilon\in\{\pm1\}
\ \middle|\
\begin{array}{l}
\exists\, h:E(J)\to E(J)\text{ a diffeomorphism such that}\\
h_*(\mu_J)=\mu_J,\quad h_*(\lambda_J)=\lambda_J^\varepsilon
\end{array}
\right\}.
\]
The diffeomorphism \(h\) is allowed to be orientation-preserving or
orientation-reversing.  We write
\[
\Gamma_\mu(J)\cdot[a]
=
\{\varepsilon[a]\mid \varepsilon\in\Gamma_\mu(J)\}.
\]
\end{definition}
For a subset \(S\subset H_1(\Sigma_g;\mathbb Z)\), we write
\[
\Stab_{\Mod(\Sigma_g)}(S)
=
\{f\in\Mod(\Sigma_g)\mid f_*(S)=S\}
\]
for its setwise stabilizer.

The group \(\Gamma_\mu(J)\) records the only knot-exterior ambiguity which
survives in the extendable-subgroup calculation.  If
\(\Gamma_\mu(J)=\{1\}\), the oriented rim class \([a]\) is fixed.  If
\(-1\in\Gamma_\mu(J)\), then the knot exterior allows the rim class to be
reversed.  The main theorem says that no other ambiguity occurs.

\begin{introtheorem}
    \label{thm:main} 
Let \(g\ge3\). Let \(\Sigma_g^0\subset S^4\) be the standard unknotted closed oriented surface, let \(a\subset\Sigma_g^0\) be an oriented nonseparating curve, and let \(J\subset S^3\) be a nontrivial knot. Let \(\Sigma_{g,a,J}\subset S^4\) be the ordinary untwisted rim-surgery surface along \(a\). Then, under the canonical marking by \(\Sigma_g^0\), \[ E(\Sigma_{g,a,J}) = \operatorname{Stab}_{\operatorname{Mod}(\Sigma_g)}(q_0) \cap \operatorname{Stab}_{\operatorname{Mod}(\Sigma_g)}(\Gamma_\mu(J)\cdot[a]). \] 
\end{introtheorem}

Thus the effect of ordinary single-rim surgery on Hirose's unknotted
extendable subgroup is completely explicit:
\[
\Stab(q_0)
\quad\leadsto\quad
\Stab(q_0)\cap\Stab(\Gamma_\mu(J)\cdot[a]).
\]
The rim-surgery complement detects the rim homology class, while the knot
exterior contributes exactly the longitude-sign ambiguity encoded by
\(\Gamma_\mu(J)\).

The substantive step in Theorem~\ref{thm:main} is the reverse inclusion, proved in Proposition~\ref{prop:lower}. It requires more than Hirose's theorem and the complement calculation.  The latter gives the upper bound in
Proposition~\ref{prop:upper}: an extendable mapping class must preserve \(q_0\)
and carry \([a]\) to a sign allowed by the peripheral symmetry of \(J\).
Conversely, Hirose's theorem extends an admissible mapping class over the
unknotted pair, but this extension carries the rim torus associated to \(a\)
to the rim torus associated to \(f(a)\).  To obtain a self-diffeomorphism of
the rim-surgered pair, one must identify the corresponding drilled exteriors
relative to the fixed surface boundary and carry the rim-longitude,
surface-meridian, and rim-meridian directions
\[
(\alpha,\mu_\Sigma,\delta)
\]
exactly, with no additional shear, so that the maps glue to the
knot-exterior piece.  This is supplied by the marked homology-relative rim
lemma, Proposition~\ref{prop:homology-relative}, and its signed version,
Proposition~\ref{prop:signed-homology-relative}.

The same methods give a classification theorem for ordinary single-rim pairs.
This also strengthens the subgroup computation by identifying exactly when two
surfaces obtained by one ordinary rim surgery are diffeomorphic as oriented
pairs.

\begin{introtheorem}
\label{thm:intro-pair-classification}
Let \(g\ge3\).  Let \(a,b\subset\Sigma_g\) be oriented nonseparating curves,
let \(J,K\subset S^3\) be nontrivial oriented knots, and let
\(f\in\Mod(\Sigma_g)\).  Then there exists an orientation-preserving
diffeomorphism
\[
F:(S^4,\Sigma_{g,a,J})
\longrightarrow
(S^4,\Sigma_{g,b,K})
\]
of \(S^4\), restricting orientation-preservingly to the surface and inducing
\(f\) under the canonical markings, if and only if there exist
\(\varepsilon\in\{\pm1\}\) and a possibly orientation-reversing
diffeomorphism
\[
h:E(J)\to E(K)
\]
such that
\[
f^*q_0=q_0,
\qquad
f_*[a]=\varepsilon[b],
\]
and
\[
h_*(\mu_J)=\mu_K,
\qquad
h_*(\lambda_J)=\lambda_K^\varepsilon.
\]
\end{introtheorem}

Consequently, after forgetting the induced mapping class, the pairs
\[
(S^4,\Sigma_{g,a,J})
\quad\text{and}\quad
(S^4,\Sigma_{g,b,K})
\]
are orientation-preservingly diffeomorphic if and only if there exist
\(f\), \(\varepsilon\), and \(h\) satisfying the displayed conditions.
When \(K=J\) and \(b=a\), the statement for the given mapping class \(f\)
says that \(f\) extends exactly when
\[
f^*q_0=q_0
\quad\text{and}\quad
f_*[a]=\varepsilon[a]
\quad\text{for some }\varepsilon\in\Gamma_\mu(J).
\]
This is Theorem~\ref{thm:main}.

If \(\Gamma_\mu(J)=\{1\}\), then Theorem~\ref{thm:main} specializes to
\[
E(\Sigma_{g,a,J})
=
\Stab(q_0)\cap\Stab([a]).
\]
This case occurs for many knots.  For example, hyperbolic knots with trivial
full symmetry group have \(\Gamma_\mu(J)=\{1\}\).  By Mostow--Prasad rigidity,
the outer automorphism induced by a complement diffeomorphism is represented by
a symmetry and is therefore trivial, so after choosing a basepoint the induced
automorphism is inner.  If conjugation by \(g\in G_J\) fixes \(\mu_J\), then
\(g\in C_{G_J}(\mu_J)\).  By the centralizer theorem for peripheral
parabolics \cite[Proposition~3.1(5)]{Friedl2011}, this centralizer is
abelian; it contains the peripheral subgroup \(P_J\).  Hence \(g\) commutes
with \(\lambda_J\), so the conjugation fixes \(\lambda_J\).  Akbulut--Ruberman give
explicit examples, including \(11n42\), \(12n0313\), and \(12n0430\)
\cite[Proposition~2.6]{AkbulutRuberman2016}.  Baker--Luecke construct further
examples among asymmetric hyperbolic \(L\)-space knots in \(S^3\)
\cite{BakerLuecke2020}.

In subsequent work, the author uses iterated rim surgery and nonabelian
exterior-group rigidity to construct, for every genus at least three,
surfaces in \(S^4\), and more generally in closed simply connected smooth
\(4\)-manifolds, whose full orientation-preserving extendable mapping-class
subgroups are trivial in both the smooth and topological categories
\cite{NiuTrivialExtendable2026}.

\medskip
\noindent\textbf{Organization.}
Section~\ref{sec:preliminaries} fixes the conventions for extendable mapping
classes, the Rokhlin quadratic form, and ordinary untwisted rim surgery.
Section~\ref{sec:complement} computes the rim-surgery complement and the
induced map from loops on the surface boundary into the knot group.
Section~\ref{sec:upper} proves the upper bound in Theorem~\ref{thm:main} by
combining this complement calculation with preservation of the Rokhlin form
and Waldhausen's peripheral realization theorem for knot exteriors.  The next three sections prove the lower-bound input.
Section~\ref{sec:local-commutator} establishes the local product-framed
commutator move.  Section~\ref{sec:handle-charts} places the required
product-framed handles in the fixed marked exterior.
Section~\ref{sec:homology-relative} combines the local move with the
Hatcher--Margalit connectivity theorem for nonseparating curves in a fixed
primitive homology class to prove the marked homology-relative rim lemma and
its signed version.  Section~\ref{sec:lower} proves the lower bound in
Theorem~\ref{thm:main}. Section~\ref{sec:proofs-main} first deduces Theorem~\ref{thm:main} from the
upper and lower bounds, then records the infinite-index consequence, and
finally proves the pair-classification theorem,
Theorem~\ref{thm:intro-pair-classification}.

\medskip
\par\noindent\textbf{Acknowledgement of AI use.}
The author used Prism, Overleaf AI, and Grammarly to assist with LaTeX formatting and English-language editing. DeepSeek was used for proofreading and the preparation of Figure 2.

\section{Preliminaries and notation}\label{sec:preliminaries}

\subsection{Extendable mapping classes}

Let \(\Sigma\subset X^4\) be a closed oriented embedded surface. Throughout, \(\Diff^+(X,\Sigma)\) means ambient orientation-preserving diffeomorphisms of \(X\) whose restriction to \(\Sigma\) is orientation-preserving. Define
\[
E(\Sigma,X)=
\im\left(
\pi_0\Diff^+(X,\Sigma)\to \Mod(\Sigma)
\right).
\]
In this paper \(X=S^4\), and \(\Sigma\) will be either the standard unknotted surface \(\Sigma_g^0\) or a rim-surgered surface \(\Sigma_{g,a,J}\).

\subsection{The Rokhlin quadratic form}

For an oriented embedded surface \(\Sigma\subset S^4\), the ambient spin structure and the oriented normal bundle induce a spin structure on \(\Sigma\). This spin structure determines a quadratic refinement
\[
q_\Sigma:H_1(\Sigma;\mathbb Z_2)\to \mathbb Z_2
\]
of the mod-two intersection pairing. We call this the Rokhlin quadratic form. For the standard unknotted surface \(\Sigma_g^0\subset S^4\), we write
\(
q_0=q_{\Sigma_g^0}.
\)
Hirose proved that
\(
E(\Sigma_g^0)=\Stab_{\Mod(\Sigma_g)}(q_0).
\)
See \cite{Hirose2002}.
\subsection{Rim surgery convention}

We recall the local-pair model of ordinary untwisted rim-surgery. The reader can refer to \cite{FintushelStern1997} for more details. Let
\(
a\subset \Sigma_g^0
\)
be an oriented nonseparating curve. Choose a product neighborhood of \(a\) in the pair
\[
(S^4,\Sigma_g^0)
\]
of the form
\[
S^1_a\times (B^3,I),
\]
where \(I\subset B^3\) is a standard properly embedded unknotted arc and the surface is locally
\[
S^1_a\times I.
\]
Given a knot \(J\subset S^3\), let \(J_+\subset B^3\) denote the corresponding knotted arc, i.e.~a 1-string tangle whose closure is \(J\). The rim-surgered surface
\[
\Sigma_{g,a,J}\subset S^4
\]
is obtained by replacing the local annulus
\[
S^1_a\times I
\]
by
\[
S^1_a\times J_+
\]
inside the same ambient block \(S^1_a\times B^3\), and leaving the surface unchanged outside this block.

This local replacement determines a canonical marking
\[
\iota:\Sigma_g^0\to \Sigma_{g,a,J}.
\]
It is the identity outside the surgery annulus and, inside the annulus, identifies
\[
S^1_a\times I\longrightarrow S^1_a\times J_+
\]
by preserving the \(S^1_a\)-coordinate and the arc parameter. This marking is an abstract parametrization of the surface; it is not an ambient isotopy between the two embedded surfaces.  Within the fixed ordinary untwisted local-pair model, two orientation-compatible arc parametrizations, product charts, or sufficiently small replacement regions can be joined by an isotopy of tubular-neighborhood germs that is stationary on a collar of the replacement boundary.  Hirsch's ambient tubular-neighborhood theorem \cite[Chapter~8, Theorem~1.8]{Hirsch1976} then extends this germ isotopy to an ambient isotopy relative to the replacement boundary.  Hence the resulting abstract surface parametrizations are isotopic relative to the replacement boundary and determine the same element of \(\Mod(\Sigma_g)\).

Equivalently, the complement of the locally modified surface is described as follows. Let
\[
M=S^4\setminus \nu\Sigma_g^0,
\qquad
\partial M=\Sigma_g\times S^1_\beta,
\]
where \(\beta=\mu_\Sigma\) is the surface meridian. The rim torus is
\[
R_a=a\times S^1_\beta\subset M,
\]
pushed slightly into a collar of \(\partial M\). We call
\[
A_a=M\setminus \nu R_a
\]
the drilled exterior. It has two boundary pieces. We call
\(\partial M=\Sigma_g\times S^1_\beta\) the outer, or surface, boundary, and
we call \(Q_a=\partial\nu R_a\cong T^3\) the internal, or rim-torus, boundary.
The rim-surgery gluing is performed along \(Q_a\), while \(\partial M\) remains
the boundary of the final surface complement.

The internal boundary
\[
Q_a=\partial\nu R_a\cong T^3
\]
has ordered basis
\[
(\alpha,\beta,\delta),
\]
where \(\alpha\) is the \(a\)-direction, \(\beta=\mu_\Sigma\), and \(\delta\) is the rim-torus meridian. We regard the drilled exterior \(A_a\) together with the fixed
outer-boundary identification
\[
\partial M=\Sigma_g\times S^1_\beta
\]
and the ordered internal basis \((\alpha,\beta,\delta)\).  A diffeomorphism
between drilled exteriors is said to preserve these marked data when it is the
identity on \(\partial M\) and carries the ordered internal basis exactly as
specified. The complement of \(\Sigma_{g,a,J}\) is obtained by gluing
\[
S^1_s\times E(J)
\]
to \(A_a\) by
\[
s\mapsto\alpha,\qquad
\mu_J\mapsto\beta,\qquad
\lambda_J\mapsto\delta.
\]
This is the standard equivalence between the local rim-surgery description and
the knot-surgery description along the rim torus.

\begin{remark}
Throughout this paper, rim surgery means ordinary untwisted rim surgery. Twisted rim-surgery variants introduce an additional meridional shear in the gluing.  Such variants change the boundary formula in the complement
calculation and would require a boundary-shear version of the
homology-relative rim lemma used in the lower-bound construction.  We do not
pursue them here.
\end{remark}

\section{The complement calculation}\label{sec:complement}
This section computes the drilled exterior and the surface-boundary map used
in the upper-bound argument.

\begin{lemma}\label{lem:drilled-exterior}
Let \(A_a=M\setminus\nu(a\times\mu_\Sigma)\). Then
\[
\pi_1(A_a)\cong
\langle \beta,\delta\mid [\beta,\delta]=1\rangle.
\]
The internal boundary map \(\pi_1(Q_a)\to\pi_1(A_a)\) is
\[
\alpha\mapsto 1,
\qquad
\beta\mapsto\beta,
\qquad
\delta\mapsto\delta.
\]
For a loop \(c\subset\Sigma_g\), the outer boundary map
\(\pi_1(\partial M)\to\pi_1(A_a)\) is
\(
c\mapsto \delta^{a\cdot c}.
\)
\end{lemma}

\begin{proof}
The standard unknotted surface is the boundary of a standard embedded
handlebody in \(S^4\).  Thus
\(
\pi_1(M)\cong\mathbb Z\langle\beta\rangle,
\)
and the surface subgroup maps trivially.  Push
\(R_a=a\times S^1_\beta\) into the collar
\(\Sigma_g\times S^1_\beta\times[0,1]\).  The drilled collar is
\[
C_a=
\left(\Sigma_g\times S^1_\beta\times[0,1]\right)
\setminus\nu(a\times S^1_\beta\times\{1/2\}),
\]
and, since the drilling is product in the \(\beta\)-direction,
\(
C_a\cong Y_a\times S^1_\beta,
\)
where
\[
Y_a=(\Sigma_g\times[0,1])\setminus\nu(a\times\{1/2\}).
\]
The element \(\delta\) is the positive meridian of the drilled curve in
\(Y_a\).

We compute the quotient of \(\pi_1(Y_a)\) obtained by killing the upper
surface group.  Let
\[
F=\overline{\Sigma_g\setminus\nu a}\cong\Sigma_{g-1,2},
\]
and denote its two boundary curves by \(a_+\) and \(a_-\).  The manifold
\(Y_a\) is the union of \(F\times[0,1]\) and the product of \(a\) with an
annulus obtained by deleting a disk from the normal rectangle to
\(a\times\{1/2\}\).  Their intersection has two annular components.  Choose
paths in \(F\) from the basepoint to the two boundary curves, and choose the
connecting path through the upper part of
the local annulus.  Van Kampen then gives
\[
\pi_1(Y_a)\cong
\left\langle
\pi_1(F),\alpha,\delta,t
\ \middle|\
[\alpha,\delta]=1,\ a_+=\alpha,\ t a_-t^{-1}=\alpha
\right\rangle.
\]
Here \(t\), together with \(\pi_1(F)\), generates the subgroup represented
by \(\Sigma_g\times\{1\}\).  Therefore
\[
\frac{\pi_1(Y_a)}
{\langle\!\langle\pi_1(\Sigma_g\times\{1\})\rangle\!\rangle}
\cong
\langle\delta\rangle\cong\mathbb Z.
\]
When the non-collar part of \(M\) is attached to the upper boundary, this
upper surface group is killed and the \(S^1_\beta\)-factor maps to \(\beta\).
Consequently
\[
\pi_1(A_a)\cong
\langle\beta,\delta\mid[\beta,\delta]=1\rangle.
\]

On the internal boundary, \(\alpha\) is represented by the copy of
\(a_+\) in the above presentation and hence maps trivially.  The other two
coordinate circles map to the generators \(\beta\) and \(\delta\).  Thus
\[
\alpha\mapsto1,
\qquad
\beta\mapsto\beta,
\qquad
\delta\mapsto\delta.
\]

Finally, fix the outer basepoint and use paths in the annulus
\(c\times[0,1]\) to the upper boundary.  After making an oriented loop
\(c\subset\Sigma_g\) transverse to \(a\), delete small disks around the
intersection points of \(c\times[0,1]\) with
\(a\times\{1/2\}\).  The upper copy of \(c\) lies in the subgroup that is
killed.  Each new boundary circle is a conjugate of \(\delta\) or
\(\delta^{-1}\), with sign equal to the local intersection sign determined by
the orientation of \(\Sigma_g\) and the positive meridian \(\delta\).  The
quotient is cyclic, so the conjugacies disappear.  The boundary relation of
the punctured annulus therefore gives
\[
c\longmapsto\delta^{\sum\operatorname{sign}(a\cap c)}
=\delta^{a\cdot c}.
\]
\end{proof}

\begin{proposition}\label{prop:complement-group}
For ordinary untwisted rim-surgery along \(a\) using \(J\), after choosing an
outer basepoint and the same paths to the internal boundary as in
Lemma~\ref{lem:drilled-exterior},
\[
\pi_1(S^4\setminus\nu\Sigma_{g,a,J})\cong G_J.
\]
Under this isomorphism,
\(
\mu_\Sigma\mapsto\mu_J,
\)
and for every based loop \(c\subset\Sigma_g\),
\(
c\mapsto \lambda_J^{a\cdot c}.
\)
\end{proposition}

\begin{proof}
The glued piece has fundamental group
\[
\pi_1(S^1_s\times E(J))=
\mathbb Z\langle s\rangle\times G_J.
\]
Use the chosen paths so that the three gluing equations are literal:
\[
s=\alpha,
\qquad
\mu_J=\beta,
\qquad
\lambda_J=\delta.
\]
By Lemma~\ref{lem:drilled-exterior}, \(\alpha=1\) in \(\pi_1(A_a)\).  Thus
van Kampen gives
\[
\pi_1(S^4\setminus\nu\Sigma_{g,a,J})
\cong
\frac{
\langle \beta,\delta\mid[\beta,\delta]=1\rangle
*
(\mathbb Z\langle s\rangle\times G_J)
}{
\langle\!\langle
s=1,\ \mu_J=\beta,\ \lambda_J=\delta
\rangle\!\rangle
}.
\]
After eliminating \(s,\beta,\delta\), the only relation imposed on \(G_J\)
is \([\mu_J,\lambda_J]=1\), which already holds in the peripheral subgroup.
Therefore the resulting group is \(G_J\).  The same based identifications give
\(
\mu_\Sigma\mapsto\mu_J
\)
and, by Lemma~\ref{lem:drilled-exterior},
\[
c\mapsto\delta^{a\cdot c}\mapsto\lambda_J^{a\cdot c}.
\]
\end{proof}
\begin{corollary}\label{cor:knotted}
If \(J\) is nontrivial, then \(\Sigma_{g,a,J}\) is not the standard unknotted surface.
\end{corollary}

\begin{proof}
The standard unknotted complement has group
\(
\pi_1(S^4\setminus\nu\Sigma_g^0)\cong\mathbb Z.
\)
By Proposition~\ref{prop:complement-group},
\(
\pi_1(S^4\setminus\nu\Sigma_{g,a,J})\cong G_J.
\)
If \(J\) is nontrivial, then \(G_J\not\cong\mathbb Z\) by the unknotting theorem for knot groups \cite[Chapter~4, Section~B, Theorem~1]{Rolfsen1976}. Hence the complements are not diffeomorphic, so \(\Sigma^{0}_{g}\) and \(\Sigma_{g,a,J}\) are not isotopic.
\end{proof}

\section{The upper bound}\label{sec:upper}

We combine the complement calculation with spin and peripheral data to prove
the upper bound.

\begin{lemma}\label{lem:rokhlin}
Under the canonical marking of \(\Sigma_{g,a,J}\) by \(\Sigma_g^0\), ordinary
untwisted rim-surgery satisfies
\(
q_{\Sigma_{g,a,J}}=q_0.
\)
\end{lemma}

\begin{proof}
The comparison is supported in the local replacement
\[
S^1\times(B^3,I)
\leadsto
S^1\times(B^3,J_+).
\]
The canonical marking preserves the \(S^1\)-coordinate and the arc parameter,
and the two surfaces, together with their boundary normal trivializations,
agree outside the replacement annulus.  It is therefore enough to compare the
induced spin structures on
\(
S^1\times I
\)
and
\(
S^1\times J_+
\)
relative to this fixed boundary trivialization.

Close \(J_+\) by the fixed trivial arc in the complementary ball used in the
local-pair model, so that the resulting closed curve is \(J\).  Extend the
product endpoint framing of \(J_+\) over this closing arc.  Close the chosen
normal pushoff of \(J_+\) by the corresponding product pushoff of the closing
arc.  The relative framing integer of the knotted arc is the linking number of
these two closed curves: changing the relative arc framing by one full twist
changes that linking number by one, while the product framing on the closing
arc is fixed.  In ordinary untwisted rim surgery the closed pushoff is the
preferred Seifert longitude \(\lambda_J\).  Hence
\[
\operatorname{lk}(J,\lambda_J)=0
\]
shows that the normal framing of \(J_+\) is homotopic rel endpoints to the
product framing of the straight arc \(I\).

The difference of the two relative spin structures on the annulus lies in
\[
H^1(S^1\times I,\partial(S^1\times I);\mathbb Z_2)
\cong\mathbb Z_2.
\]
Its value on a properly embedded arc transverse to the \(S^1\)-direction is
the mod-two relative winding of the normal framing along the arc factor.  The
preceding framed comparison shows that this winding is zero.  Equivalently,
under Poincar\'e--Lefschetz duality the possible difference is dual to the
annular core, and the relative framing parity just computed is its
coefficient.  Thus the difference class vanishes.  The induced spin
structures agree under the canonical marking, and therefore so do their
quadratic refinements:
\[
q_{\Sigma_{g,a,J}}=q_0.
\]
\end{proof}

\begin{remark}
The above proof uses the untwisted convention in an essential way. If one changes the gluing by introducing a twist in the knot-meridian direction, then the relative spin class on the local annulus is changed by the corresponding mod-two normal-framing winding. We do not use any twisted convention in this paper.
\end{remark}

\begin{convention}[basepoints and peripheral coordinates]
\label{conv:basepoints-peripheral}
All fundamental groups of rim-surgery complements are taken with a basepoint
on the surface boundary, together with a fixed path to the internal
rim-surgery piece.  A diffeomorphism of complements naturally induces an
outer isomorphism.  When the diffeomorphism preserves the oriented surface
meridian, we choose the basepoint paths so that the corresponding representative
of the outer isomorphism sends the oriented knot meridian literally to the
oriented knot meridian.  Thus, for a self-map using \(J\), we write
\(
\mu_J\mapsto\mu_J,
\)
and for a map from the \(J\)-rim-surgery complement to the \(K\)-rim-surgery
complement, we write
\(
\mu_J\mapsto\mu_K.
\)
This convention is used only to choose representatives of the induced outer
isomorphisms for which the meridian equations are literal.  The longitude sign
is not assumed; it is determined in Proposition~\ref{prop:upper} and in
Theorem~\ref{thm:intro-pair-classification}.

For a nontrivial knot \(J\), the boundary torus of \(E(J)\) is incompressible
\cite[Chapter~4, Section~B, Theorem~2]{Rolfsen1976}, so the peripheral map is injective:
\[
\pi_1(\partial E(J))\hookrightarrow G_J.
\]
Thus
\[
P_J=\langle\mu_J,\lambda_J\rangle\cong\mathbb Z^2
\]
is an embedded free abelian subgroup of \(G_J\).  In particular, if
\[
\lambda_J^m\mu_J^n=\lambda_J^{m'}\mu_J^{n'}
\]
inside \(G_J\), then
\(
m=m'\) and  \(n=n'.
\)
\end{convention}

\begin{theorem}[peripheral realization for knot exteriors]
\label{thm:peripheral-realization}
Let \(J,K\subset S^3\) be nontrivial knots.  Let
\[
G_J=\pi_1(E(J)),\qquad G_K=\pi_1(E(K)),
\]
and let
\[
P_J=\langle\mu_J,\lambda_J\rangle,\qquad
P_K=\langle\mu_K,\lambda_K\rangle
\]
be the chosen peripheral subgroups.  Suppose \(\varphi:G_J\to G_K\) is an
isomorphism such that \(\varphi(P_J)\) is conjugate in \(G_K\) to \(P_K\).
Then, up to inner conjugacy in \(G_K\), \(\varphi\) is induced by a
possibly orientation-reversing diffeomorphism
\[
h:E(J)\longrightarrow E(K)
\]
whose boundary restriction realizes the induced peripheral action
\(
\varphi|_{P_J}:P_J\to P_K.
\)

In particular, if, after the chosen basepoint conventions,
\[
\varphi(\mu_J)=\mu_K,
\qquad
\varphi(\lambda_J)=\lambda_K^\varepsilon,
\qquad
\varepsilon\in\{\pm1\},
\]
then \(h\) may be chosen so that
\[
h_*(\mu_J)=\mu_K,
\qquad
h_*(\lambda_J)=\lambda_K^\varepsilon.
\]
\end{theorem}

\begin{proof}
A nontrivial knot exterior is compact and orientable.  It is irreducible:
a sphere in the exterior bounds a ball in \(S^3\), and the side not containing
the knot is a ball in the exterior.  Its boundary torus is incompressible by
\cite[Chapter~4, Section~B, Theorem~2]{Rolfsen1976}.  It is also sufficiently
large: a minimal-genus Seifert surface is incompressible, since a compression
would produce a Seifert surface of smaller genus, and a nontrivial knot has
positive genus.  Thus the hypotheses of Waldhausen's peripheral realization
theorem \cite[Corollary~6.5]{Waldhausen1968} apply to \(E(J)\) and \(E(K)\).

After replacing \(\varphi\) by an inner conjugate, assume
\(
\varphi(P_J)=P_K.
\)
Choose first a diffeomorphism
\[
b:\partial E(J)\longrightarrow\partial E(K)
\]
whose action on \(\pi_1(T^2)\cong\mathbb Z^2\) is
\(\varphi|_{P_J}\).  Waldhausen's theorem gives a homeomorphism of knot
exteriors inducing \(\varphi\).  Its boundary restriction and \(b\) induce
the same automorphism of \(\pi_1(T^2)\), so they are isotopic as maps of the
torus; for orientation-preserving maps this is \cite[Theorem~2.5]{FarbMargalit2012}, and the orientation-reversing case follows by composing both maps with a fixed reflection.  Extending this isotopy across a boundary collar changes the
homeomorphism so that it is exactly \(b\) on the boundary and is product with
\(b\) on a smaller collar.

The resulting homeomorphism is smooth and equal to the prescribed product map
on that collar.  The relative approximation theorem for homeomorphisms of
smooth \(3\)-manifolds gives diffeomorphisms arbitrarily close to it and equal
to it on a still smaller boundary collar
\cite[Theorem~4.6]{Pardon2021}.  Choose the approximation sufficiently close
that it is homotopic to the homeomorphism rel that smaller collar.  For
example, after fixing a Riemannian metric, the two maps are equal on the collar
and, on its compact complement, corresponding image points are joined by
unique short geodesics.  The diffeomorphism therefore induces the
same fundamental-group isomorphism and has the same exact boundary map.  The
final assertion follows by choosing \(b\) to send \(\mu_J\) to \(\mu_K\) and
\(\lambda_J\) to \(\lambda_K^\varepsilon\).
\end{proof}

\begin{proposition}[upper bound]
\label{prop:upper}
Let \(g\ge3\), and let \(J\subset S^3\) be a nontrivial knot.  Then
\[
E(\Sigma_{g,a,J})
\subseteq
\Stab_{\Mod(\Sigma_g)}(q_0)
\cap
\Stab_{\Mod(\Sigma_g)}(\Gamma_\mu(J)\cdot[a]).
\]
\end{proposition}

\begin{proof}
Let
\(
F:(S^4,\Sigma_{g,a,J})\to(S^4,\Sigma_{g,a,J})
\)
be an orientation-preserving pair diffeomorphism, and let
\(
f\in\Mod(\Sigma_g)
\)
be the induced mapping class under the canonical marking.

By Proposition~\ref{prop:complement-group}, the complement group is identified
with
\(
G_J=\pi_1(E(J)).
\)
Using Convention~\ref{conv:basepoints-peripheral}, the induced automorphism
\(
\varphi:G_J\to G_J
\)
satisfies
\(
\varphi(\mu_J)=\mu_J.
\)

By uniqueness of tubular neighborhoods and isotopy extension
\cite[Chapter~4, Theorem~5.3 and Chapter~8, Theorem~1.3]{Hirsch1976}, after an
isotopy through pair diffeomorphisms supported in a tubular neighborhood of
the surface and fixed on the surface, assume that \(F\) is an oriented
normal-bundle map near \(\Sigma_{g,a,J}\).  This isotopy does not change the
induced mapping class \(f\) or the induced outer automorphism of the
complement group.  There is no ambient boundary condition to preserve in this
normalization.

The resulting restriction of \(F\) to the surface-boundary circle bundle
\[
\partial\nu\Sigma_{g,a,J}\cong \Sigma_g\times S^1_{\mu_\Sigma}
\]
preserves the oriented meridian circle.  Hence, for every based loop
\(c\in\pi_1(\Sigma_g)\), viewed as a loop in the surface factor of
\(\Sigma_g\times S^1_{\mu_\Sigma}\), we have
\[
F_*(c)=f_*(c)\mu_\Sigma^{u([c])}
\]
in
\(
\pi_1(\Sigma_g\times S^1_{\mu_\Sigma})
\cong
\pi_1(\Sigma_g)\times \langle\mu_\Sigma\rangle,
\)
for some homomorphism
\[
u:H_1(\Sigma_g;\mathbb Z)\to\mathbb Z.
\]
Using the boundary formula of Proposition~\ref{prop:complement-group},
naturality gives
\[
\varphi(\lambda_J^{a\cdot c})
=
\lambda_J^{a\cdot f_*(c)}\mu_J^{u(c)}
\]
for every loop \(c\subset\Sigma_g\). Choose \(c_0\subset\Sigma_g\) with
\(
a\cdot c_0=1.
\)
Then
\[
\varphi(\lambda_J)
=
\lambda_J^{a\cdot f_*(c_0)}\mu_J^{u(c_0)}
\in P_J.
\]
Since
\(
\varphi(\mu_J)=\mu_J,
\)
we have
\(
\varphi(P_J)\subseteq P_J.
\)

Apply the same argument to \(F^{-1}\).  Its induced automorphism is
\(\varphi^{-1}\), and it also preserves the oriented meridian.  Therefore
\(
\varphi^{-1}(P_J)\subseteq P_J.
\)
Equivalently,
\(
P_J\subseteq \varphi(P_J).
\)
Hence
\(
\varphi(P_J)=P_J.
\)

Thus \(\varphi|_{P_J}\) is an automorphism of the free abelian group
\[
P_J=\langle\mu_J,\lambda_J\rangle\cong\mathbb Z^2
\]
which fixes \(\mu_J\).  Therefore
\(
\varphi(\lambda_J)=\mu_J^k\lambda_J^\varepsilon
\)
for some
\[
k\in\mathbb Z,
\qquad
\varepsilon\in\{\pm1\}.
\]
Since
\(
H_1(E(J);\mathbb Z)\cong\mathbb Z\langle\mu_J\rangle
\)
and
\(
[\lambda_J]=0\in H_1(E(J);\mathbb Z),
\)
we have
\[
0
=
\varphi_*[\lambda_J]
=
[\mu_J^k\lambda_J^\varepsilon]
=
k[\mu_J].
\]
Thus
\(
k=0,
\)
and therefore
\(
\varphi(\lambda_J)=\lambda_J^\varepsilon.
\)

By Theorem~\ref{thm:peripheral-realization}, this peripheral action is realized
by a diffeomorphism
\[
h:E(J)\to E(J)
\]
with
\(
h_*(\mu_J)=\mu_J\) and \(
h_*(\lambda_J)=\lambda_J^\varepsilon.
\)
Hence
\(
\varepsilon\in\Gamma_\mu(J).
\)
Since
\(
\varphi(\lambda_J)=\lambda_J^\varepsilon,
\)
we obtain
\[
\lambda_J^{\varepsilon(a\cdot c)}
=
\lambda_J^{a\cdot f_*(c)}\mu_J^{u(c)}.
\]
By Convention~\ref{conv:basepoints-peripheral}, the peripheral subgroup
\(
P_J=\langle\mu_J,\lambda_J\rangle\cong\mathbb Z^2
\)
injects into \(G_J\), so the \(\mu_J\)- and \(\lambda_J\)-coordinates are
independent.  Hence
\(
u(c)=0\) and \(
a\cdot f_*(c)=\varepsilon(a\cdot c)
\)
for all \(c\). Since \(f_*\) preserves the algebraic intersection form,
\(
a\cdot f_*(c)=f_*^{-1}(a)\cdot c.
\)
By nondegeneracy of the intersection pairing,
\(
f_*^{-1}[a]=\varepsilon[a],
\)
and therefore
\(
f_*[a]=\varepsilon[a].
\)
Thus
\[
f\in\Stab_{\Mod(\Sigma_g)}(\Gamma_\mu(J)\cdot[a]).
\]

Finally, \(F\) preserves the induced spin structure on the surface.  By
Lemma~\ref{lem:rokhlin}, ordinary untwisted rim-surgery has
\(
q_{\Sigma_{g,a,J}}=q_0
\)
under the canonical marking.  Hence
\(
f^*q_0=q_0.
\)
Therefore
\[
f\in
\Stab_{\Mod(\Sigma_g)}(q_0)
\cap
\Stab_{\Mod(\Sigma_g)}(\Gamma_\mu(J)\cdot[a]).
\]
\end{proof}

\section{The local commutator move}\label{sec:local-commutator}

Let
\[
P\cong\Sigma_{1,2},\qquad \partial P=a\cup(-a'),
\]
and choose oriented embedded simple closed curves \(x,y\subset P\) meeting
transversely in exactly one point \(q\), with positive local intersection
sign.  In particular,
\(
x\cdot y=1.
\)
In the global application \(P\) will be an embedded subsurface of
\(\Sigma_g\).  In this section, it is treated as an abstract model.

\begin{convention}[boundary word and endpoint collars]
\label{conv:local-boundary-word}
Let
\[
Q=N(x\cup y)\subset\operatorname{int}P
\]
be an oriented regular neighborhood.  Then \(Q\cong\Sigma_{1,1}\).  Choose a
basepoint \(p\in\partial Q\) and disjoint paths from \(p\) to \(x\) and
\(y\).  With these orientations, the boundary of
\(Q\), oriented as the boundary of \(Q\), represents
\[
\partial Q=[x,y]=xyx^{-1}y^{-1}\in\pi_1(Q,p).
\]
Let
\[
\Pi=\overline{P\setminus Q}.
\]
Thus \(\Pi\) is a pair of pants.  We choose paths in \(\Pi\) from \(p\) to
\(a\) and \(a'\).  The notation
\[
a'=[x,y]a
\]
means precisely that, with these chosen paths and with the above orientation of
\(\partial Q\), the pair-of-pants relation identifies the based loop
represented by \(a'\) with the product of the based loop \(\partial Q=[x,y]\)
and the based loop represented by \(a\).  Geometrically, \(a'\) is obtained
from \(a\) across the pair of pants \(\Pi\) by inserting the boundary loop of
the genus-one piece \(Q\).

Whenever \(a\) and \(a'\) define rim tori in this local model, fixed collar
pushoffs of the boundary components into \(P\) are understood.  Equivalently,
in the global application one works in a collar enlargement
\(P\subset\operatorname{int}P^+\), and the tori
\[
a\times S^1_\beta,\qquad a'\times S^1_\beta
\]
are taken in \(P^+\times S^1_\beta\).  The product meridian framing is the
ordered normal framing given by the surface-normal direction in \(P^+\) and
the collar direction in \(P^+\times S^1_\beta\times[0,1]\).
\end{convention}

Let
\[
W_{\mathrm{loc}}
=
P\times S^1_\beta\times[0,1]\cup h_x\cup h_y,
\]
where \(h_x\) and \(h_y\) are relative \(2\)-handles attached to
\(P\times S^1_\beta\times\{1\}\) along
\[
x\times\{0\},\qquad y\times\{1/2\},
\]
respectively.  Here \(0,1/2\in S^1_\beta=\mathbb R/\mathbb Z\).  We call 
\(
P\times S^1_\beta\times\{0\}
\)
the outer boundary.
\begin{figure}[ht]
\centering
\begin{tikzpicture}[scale=0.78, every node/.style={font=\small}]

    \draw[line width=0.8pt] (0,0) rectangle (6,6);

    \fill[gray!18] (0,2.6) rectangle (6,3.4);
    \fill[gray!18] (2.6,0) rectangle (3.4,6);

    \fill[white] (1.7,4.3) circle (0.34);
    \draw[line width=0.8pt] (1.7,4.3) circle (0.34);

    \fill[white] (4.3,1.7) circle (0.34);
    \draw[line width=0.8pt] (4.3,1.7) circle (0.34);

    \draw[blue, line width=1.0pt] (0,3) -- (6,3);
    \draw[red, line width=1.0pt] (3,0) -- (3,6);

    \node[blue] at (5.45,3.23) {$x$};
    \node[red] at (3.22,5.45) {$y$};

    \node at (1.00,4.85) {$a$};
    \node at (4.85,0.95) {$a'$};

    \node[gray!70!black] at (4.75,4.55) {$Q=\Sigma_{1,1}$};
    \draw[->, line width=0.5pt, draw=gray!60] (4.38,4.28) -- (3.78,3.78);

    \node at (3,-0.62) {$P\cong \Sigma_{1,2}$};

    \draw[->, line width=0.7pt] (2.1,6.20) -- (3.9,6.20);
    \draw[->, line width=0.7pt] (2.1,-0.20) -- (3.9,-0.20);

    \draw[->, line width=0.7pt] (-0.20,2.1) -- (-0.20,3.9);
    \draw[->, line width=0.7pt] (6.20,2.1) -- (6.20,3.9);
\end{tikzpicture}
\caption{A square model of the genus-one subsurface \(P\cong \Sigma_{1,2}\), with boundary components \(a\) and \(a'\), standard curves \(x\) and \(y\), and a regular neighborhood \(Q=\Sigma_{1,1}\) of \(x\cup y\). Opposite sides are identified according to the arrows.}
\label{fig:sigma12-Q}
\end{figure}
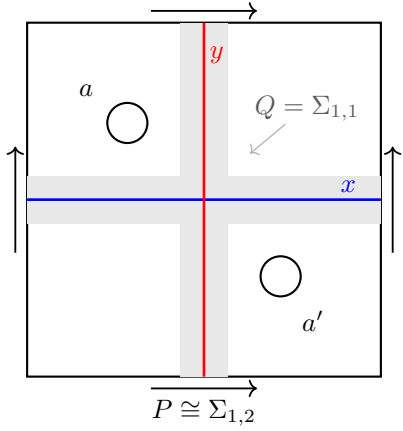
\begin{convention}[product-framed surgery]\label{conv:product-framed-surgery}
For a curve
\[
z\times\{\theta_z\}\subset P\times S^1_\beta,
\]
the product framing is the framing whose longitude is the surface pushoff of
\(z\subset P\) at the same \(\beta\)-level.  Product-framed surgery means Dehn
surgery whose filling slope is this product longitude.
\end{convention}

\begin{lemma}[surgery on \(Q\times S^1_\beta\)]\label{lem:Q-surgery-solid-torus}
Let \(x,y\subset P\) be oriented embedded simple closed curves meeting
transversely in exactly one point with positive local intersection sign, and
let \(Q=N(x\cup y)\subset P\).  Then \(Q\cong\Sigma_{1,1}\),
\(
x\cdot y=1,
\)
and
\(
\partial Q=[x,y].
\)
Product-framed surgery on
\[
x\times\{0\},\qquad y\times\{1/2\}
\subset Q\times S^1_\beta
\]
gives
\(
D^2_{\partial Q}\times S^1_\beta,
\)
where the meridian is \(\partial Q=[x,y]\) and the longitude is the
\(S^1_\beta\)-factor.
\end{lemma}

\begin{proof}
Use the standard Kirby diagram for \(T^3\): the Borromean rings
\[
B_x\cup B_y\cup B_\beta\subset S^3,
\]
with framing \(0\) on all three components.  Let
\(\mu_i,\lambda_i\subset\partial\nu B_i\), \(i=x,y,\beta\), denote the
meridian and the preferred \(0\)-framed longitude, respectively. Performing \(0\)-surgery on
all three components gives \(T^3\).

First perform the \(0\)-fillings on \(B_x\) and \(B_y\), leaving the
\(B_\beta\)-boundary unfilled.  The result is
\[
E_B(\lambda_x,\lambda_y),
\qquad
E_B=S^3\setminus \nu(B_x\cup B_y\cup B_\beta).
\]
Equivalently, this is \(T^3\) with a tubular neighborhood of the core circle
of the \(B_\beta\)-filling removed.  Identifying that core circle with the
\(\beta\)-coordinate circle in \(T^3=T^2_{x,y}\times S^1_\beta\), we obtain
\[
E_B(\lambda_x,\lambda_y)
\cong
(T^2_{x,y}\setminus D^2)\times S^1_\beta
=
Q\times S^1_\beta.
\]
Under this identification, the core circles of the \(B_x\)- and
\(B_y\)-fillings are
\[
K_x=x\times\{0\},
\qquad
K_y=y\times\{1/2\},
\]
after separating them by the indicated \(\beta\)-levels.

For the \(\lambda_x\)-filling of \(B_x\), the meridian of the core circle
\(K_x\) is
\(
m_{K_x}=\lambda_x,
\)
while its longitude is
\(
\ell_{K_x}=\mu_x.
\)
In the product identification \(Q\times S^1_\beta\), this longitude is the
surface pushoff of \(x\times\{0\}\), hence is exactly the product framing.
Similarly,
\(
m_{K_y}=\lambda_y\) and \(
\ell_{K_y}=\mu_y,
\)
and \(\ell_{K_y}\) is the product longitude of \(y\times\{1/2\}\).

Therefore product-framed surgery on \(K_x\) and \(K_y\) corresponds, in the
Borromean exterior, to filling the \(B_x\)- and \(B_y\)-boundary tori along
\(
\mu_x\) and \(
\mu_y.
\)
But filling along \(\mu_x\) and \(\mu_y\) glues back the original solid tori
\(\nu B_x\) and \(\nu B_y\).  Hence the surgered manifold is
\[
E_B(\mu_x,\mu_y)
\cong
S^3\setminus \nu B_\beta.
\]
Since \(B_\beta\) is an unknot,
\[
S^3\setminus \nu B_\beta\cong D^2\times S^1.
\]

It remains to identify the boundary slopes.  The remaining boundary torus is
\(\partial\nu B_\beta\).  In the \(T^3\)-description, the missing
\(\lambda_\beta\)-filling caps the boundary circle of
\(T^2_{x,y}\setminus D^2\).  Thus, after orienting \(B_\beta\) consistently,
\(
\lambda_\beta=\partial Q=[x,y].
\)
The meridian \(\mu_\beta\) of \(B_\beta\) is the surviving coordinate circle,
namely
\(
\mu_\beta=\beta.
\)
For the unknot exterior \(S^3\setminus\nu B_\beta\), the meridian disk of the
solid torus has boundary \(\lambda_\beta\), and the core longitude is
\(\mu_\beta\).  Therefore the product-framed surgery result is
\(
D^2_{\partial Q}\times S^1_\beta,
\)
with meridian \(\partial Q=[x,y]\) and longitude \(\beta\).
\end{proof}

\begin{lemma}[surgery on \(P\times S^1_\beta\)]\label{lem:P-surgery-product}
Let \(N_0=P\times S^1_\beta\). Performing product-framed surgery on
\[
x\times\{0\},\qquad y\times\{1/2\}
\]
turns \(N_0\) into
\[
A\times S^1_\beta\cong T^2\times I,
\]
where \(A\) is an annulus with boundary
\(
\partial A=a\cup(-a').
\)
The two boundary components are
\(
a\times S^1_\beta\) and \(
a'\times S^1_\beta.
\)
\end{lemma}

\begin{proof}
Let \(Q=N(x\cup y)\subset P\).  By Convention~\ref{conv:local-boundary-word},
\(\partial Q=[x,y]\) and the complementary pair of pants
\(
\Pi=\overline{P\setminus Q}
\)
identifies \(a'\) with \([x,y]a\).  Thus capping the \(\partial Q\)-boundary
component of \(\Pi\) gives an annulus \(A\) with
\(
\partial A=a\cup(-a').
\)

The decomposition
\[
P\times S^1_\beta
=
(Q\times S^1_\beta)
\cup_{\partial Q\times S^1_\beta}
(\Pi\times S^1_\beta)
\]
has the surgeries supported in \(Q\times S^1_\beta\).  By
Lemma~\ref{lem:Q-surgery-solid-torus}, the surgered copy of
\(Q\times S^1_\beta\) is
\(
D^2_{\partial Q}\times S^1_\beta.
\)
Gluing this solid torus to \(\Pi\times S^1_\beta\) caps the
\(\partial Q\)-boundary component of \(\Pi\).  Hence the surgered manifold is
\(
A\times S^1_\beta\cong T^2\times I.
\)
\end{proof}

\begin{lemma}[embedded local surgery trace]
\label{lem:product-trace-isotopy}
Fix \(r_0\in(0,1)\).  Using the collar pushoffs of
Convention~\ref{conv:local-boundary-word}, set
\[
R_{a'}=a'\times S^1_\beta\times\{r_0\},
\qquad
R_a=a\times S^1_\beta\times\{r_0\}.
\]
Then \(W_{\mathrm{loc}}\) contains a compact embedded oriented \(3\)-manifold
\[
\mathcal T\cong A\times S^1_\beta\cong T^2\times[0,1]
\]
with
\(
\partial\mathcal T=R_{a'}\sqcup(-R_a).
\)
Moreover, \(\mathcal T\) is disjoint from a collar of the local outer boundary
\(P\times S^1_\beta\times\{0\}\); near \(R_{a'}\) and \(R_a\) it agrees with
product collars in \(P\times S^1_\beta\times\{r_0\}\); and the slices of
\(\mathcal T\cong T^2\times[0,1]\) give an isotopy from \(R_{a'}\) to \(R_a\).
This isotopy extends to an ambient isotopy supported in an arbitrarily small
regular neighborhood of \(\mathcal T\), fixed near
\(P\times S^1_\beta\times\{0\}\).  Consequently,
\[
W_{\mathrm{loc}}\setminus\nu R_{a'}
\cong
W_{\mathrm{loc}}\setminus\nu R_a
\]
rel the local outer boundary $P\times S^1_{\beta}\times\{0\}$.
\end{lemma}

\begin{proof}
Write
\(
N=P\times S^1_\beta.
\)
The handles \(h_x\) and \(h_y\) are attached to \(N\times\{1\}\) along
\[
K_x=x\times\{0\}\times\{1\},
\qquad
K_y=y\times\{1/2\}\times\{1\}.
\]
Choose handle coordinates
\[
h_z=D^2_z\times D^2_{\mathrm{fr},z},
\qquad z=x,y,
\]
so that \(\partial D^2_z\times D^2_{\mathrm{fr},z}\) is identified with a
product-framed tubular neighborhood of \(K_z\) in \(N\times\{1\}\).  The
post-handle top boundary of the two-handle trace contains the Dehn-surgered
\(3\)-manifold
\[
\widehat N
=
\Bigl(N\setminus\operatorname{int}(\nu K_x\cup\nu K_y)\Bigr)
\cup
\bigl(D^2_x\times\partial D^2_{\mathrm{fr},x}\bigr)
\cup
\bigl(D^2_y\times\partial D^2_{\mathrm{fr},y}\bigr),
\]
where $\nu K_x$ and $\nu K_y$ are product neighborhoods of $K_x$ and $K_y$ in $N\times \{1\}$ respectively. The disks \(D^2_z\times\{\theta\}\), \(\theta\in\partial D^2_{\mathrm{fr},z}\),
are the core disks of the \(2\)-handle that perform the surgery.

By Lemma~\ref{lem:P-surgery-product},
\(
\widehat N\cong A_1\times S^1_\beta,
\)
where \(A_1\) is an annulus whose boundary components are the copies of
\(a'\) and \(a\) in the post-handle top boundary.  Since the surgery curves are
contained in
\(
Q\times S^1_\beta\subset \operatorname{int}P\times S^1_\beta,
\)
the surgery is supported away from fixed collar neighborhoods of the two
boundary components \(a'\) and \(a\).  Hence, on these collars, the
identification
\(
\widehat N\cong A_1\times S^1_\beta
\)
agrees with the original product structure of \(P\times S^1_\beta\).  In
particular, the boundary collars of \(A_1\) inherit the original surface-collar
coordinates and the original \(S^1_\beta\)-coordinate.

The submanifold \(A_1\times S^1_\beta\) initially lies in the post-surgery top
boundary of the two-handle trace.  Using an inward collar of this post-surgery
boundary, we push \(A_1\times S^1_\beta\) slightly into the interior of
\(W_{\mathrm{loc}}\).  Near the \(a'\)- and \(a\)-ends, this push is chosen to
be product with respect to the endpoint collar coordinates just described.
The push is chosen away from the fixed outer boundary
\(
P\times S^1_\beta\times\{0\}.
\)

We now explain explicitly how the pushed-in copy of
\(A_1\times S^1_\beta\) is joined to the prescribed level-\(r_0\) tori.
Choose disjoint collar annuli
\[
C_{a'}\cong a'\times[0,\varepsilon],
\qquad
C_a\cong a\times[0,\varepsilon]
\]
of \(a'\) and \(a\) in \(P\), with
\[
C_{a'}\cap(x\cup y)=\varnothing,
\qquad
C_a\cap(x\cup y)=\varnothing.
\]
Since the surgeries are supported in \(Q\times S^1_\beta\subset
\operatorname{int}P\times S^1_\beta\), these collars are unchanged by the
surgery and identify with collar neighborhoods of the corresponding boundary
components of the annulus \(A_1\).  Thus, near the \(a'\)- and \(a\)-ends, the
pushed-in copy of \(A_1\times S^1_\beta\) has product collars
\[
[0,\varepsilon) \times a'\times S^1_\beta,
\qquad
[0,\varepsilon) \times a\times S^1_\beta,
\]
with the original \(S^1_\beta\)-coordinate.  Let
\(
T_{a'}^{\mathrm{top}} \text{ and } T_a^{\mathrm{top}}
\)
denote the two boundary tori of the pushed-in copy of
\(A_1\times S^1_\beta\).  For each \(b\in\{a',a\}\), choose a smooth function
\[
\rho_b:[0,\varepsilon]\longrightarrow [r_0,1)
\]
which is equal to \(r_0\) near \(0\), is equal to the inward-pushed top level
near \(\varepsilon\), and is monotone on the intervening interval.  The graph
of \(\rho_b\) in the collar region \(C_b\times S^1_\beta\times[0,1]\) gives a
product ramp
\[
E_b
=
\{(p,s,\theta,\rho_b(s))\mid p\in b,\ s\in[0,\varepsilon],\
\theta\in S^1_\beta\}
\cong b\times[0,\varepsilon]\times S^1_\beta .
\]
Its lower boundary is the prescribed torus $R_b = b\times S^1_\beta\times\{r_0\}$, and its upper boundary is the corresponding boundary torus
\(T_b^{\mathrm{top}}\) of the pushed-in \(A_1\times S^1_\beta\).  Thus the
gluing is along a torus; the adjacent collar of \(T_b^{\mathrm{top}}\) inside
\(A_1\times S^1_\beta\) is used only to make the union smooth after rounding
corners.  Near \(b\times S^1_\beta\times\{r_0\}\), the function \(\rho_b\) is
constant, so \(E_b\) agrees with a horizontal product collar in
\(P\times S^1_\beta\times\{r_0\}\).  Because the collars \(C_a\) and \(C_{a'}\)
are disjoint from \(x\cup y\), the ramp pieces are disjoint from the attaching
regions of \(h_x\) and \(h_y\).  Therefore
\[
\mathcal T
=
E_{a'}\cup (A_1\times S^1_\beta)\cup E_a,
\]
after rounding corners, is a compact embedded oriented \(3\)-manifold
diffeomorphic to \(A\times S^1_\beta\).  Orient \(A\) so that
\(
\partial A=a'\sqcup(-a).
\)
Then
\(
\partial\mathcal T
=
R_{a'}\sqcup(-R_a).
\)

Embeddedness follows from the construction: the two handle-core disk families
lie in distinct \(2\)-handles, the endpoint collars are disjoint from the
attaching regions, and the pieces are glued only along the prescribed boundary
annuli.  Choose a product parametrization
\[
j:T^2\times[0,1]\longrightarrow \mathcal T
\]
with \(j(T^2\times\{0\})=R_{a'}\) and \(j(T^2\times\{1\})=R_a\).  The images of
the slices are pairwise disjoint, so they give an isotopy of embedded tori.
The trace is disjoint from a collar of \(P\times S^1_\beta\times\{0\}\), and a
regular neighborhood of \(\mathcal T\) may be chosen disjoint from that collar.
Applying isotopy extension in this regular neighborhood
\cite[Chapter~8, Theorem~1.3]{Hirsch1976} gives an ambient isotopy fixed near
the local outer boundary.  Removing tubular neighborhoods of
the initial and final tori gives the claimed diffeomorphism of drilled
exteriors.  The ambient extension used to compare the internal meridians is
chosen on the full framed tubular neighborhoods in
Lemma~\ref{lem:product-trace-basis}.
\end{proof}

\begin{lemma}[endpoint normal-framing comparison]
\label{lem:product-trace-basis}
The isotopy of the rim tori determined by \(\mathcal T\) has an ambient
extension carrying the ordered internal basis by
\[
\alpha_{a'}\mapsto\alpha_a,
\qquad
\beta\mapsto\beta,
\qquad
\delta_{a'}\mapsto\delta_a.
\]
Equivalently, the normal-framing difference class in
\(
H^1(T^2;\mathbb Z)
\)
is zero.  In particular, no terms
\[
\delta\mapsto \delta+\tau\alpha
\qquad\text{or}\qquad
\delta\mapsto \delta+\sigma\beta
\]
occur.
\end{lemma}

\begin{proof}
Write
\(
\mathcal T=A\times S^1_\beta,
\)
and choose a product parametrization whose oriented circle slices are
\(c_t\times S^1_\beta\), with \(c_0=a'\) and \(c_1=a\).  Put
\(
T_t=c_t\times S^1_\beta.
\)
Let \(n_t\) be the tangent vector field in \(A\), normal to \(c_t\), pointing
in the direction of increasing \(t\), and let \(\eta_t\) be the oriented
normal line to \(\mathcal T\subset W_{\mathrm{loc}}\), chosen so that
\[
(TT_t,n_t,\eta_t)
\]
gives the orientation of \(W_{\mathrm{loc}}\).  The pair
\((n_t,\eta_t)\) is a smooth framing of the normal \(2\)-plane bundle of the
moving tori.

Near \(R_{a'}\) and \(R_a\), Lemma~\ref{lem:product-trace-isotopy} identifies
the trace with the fixed product collars.  Hence the endpoint frames
\((n_0,\eta_0)\) and \((n_1,\eta_1)\) are exactly the product meridian
frames used to define \(\delta_{a'}\) and \(\delta_a\).  Choose a smooth
family of tubular-neighborhood embeddings
\[
e_t:T^2\times D^2\longrightarrow W_{\mathrm{loc}}
\]
whose central torus is \(T_t\), whose ordered disk-coordinate frame is sent
to \((n_t,\eta_t)\), and whose endpoint embeddings are the fixed product
tubular coordinates of \(R_{a'}\) and \(R_a\).  The neighborhoods may be
chosen uniformly small and disjoint from the protected collar of the local
outer boundary.

Apply isotopy extension to the embeddings \(e_t\) of the full tubular
neighborhood \cite[Chapter~8, Theorem~1.3]{Hirsch1976}.  Concretely, on \(e_t(T^2\times D^2)\) the vector field
\[
\frac{d}{dt}e_t\circ e_t^{-1}
\]
extends, using a cutoff, to a compactly supported vector field in the interior
of \(W_{\mathrm{loc}}\).  Its flow \(\Phi_t\) satisfies
\(
\Phi_t\circ e_0=e_t.
\)
Thus \(\Phi_1\) carries the entire product normal disk coordinate at
\(R_{a'}\) to the fixed product normal disk coordinate at \(R_a\).

The first coordinate of the product parametrization moves the oriented
\(a'\)-circle across \(A\), and the second coordinate is the unchanged
\(S^1_\beta\)-factor.  Hence
\[
\alpha_{a'}\mapsto\alpha_a,
\qquad
\beta\mapsto\beta.
\]
Since the boundary of each normal disk is carried coordinatewise,
\(
\delta_{a'}\mapsto\delta_a.
\)
A shear of the meridian would require a nontrivial rotation of the disk
coordinate by a map \(T^2\to SO(2)\); the identity
\(\Phi_1\circ e_0=e_1\) excludes such a rotation.  Therefore both possible
shear coefficients vanish.
\end{proof}

\begin{theorem}[product-framed local commutator move]\label{thm:local-commutator}
Let
\[
P\cong\Sigma_{1,2},
\qquad
\partial P=a\cup(-a'),
\]
and choose oriented embedded simple closed curves \(x,y\subset P\) meeting
transversely in exactly one point \(q\), with positive local intersection
sign, and satisfying
\[
a'=[x,y]a
\]
in the sense of Convention~\ref{conv:local-boundary-word}.  Let
\[
W_{\mathrm{loc}}
=
P\times S^1_\beta\times[0,1]\cup h_x\cup h_y,
\]
where \(h_x\) and \(h_y\) are attached along
\[
x\times\{0\},
\qquad
y\times\{1/2\},
\]
with product framings relative to \(P\times S^1_\beta\).  Then there is a
compactly supported diffeomorphism
\[
W_{\mathrm{loc}}\setminus\nu R_{a'}
\cong
W_{\mathrm{loc}}\setminus\nu R_a
\]
fixed on the outer boundary and carrying
\[
(\alpha_{a'},\beta,\delta_{a'})
\mapsto
(\alpha_a,\beta,\delta_a).
\]
\end{theorem}

\begin{proof}
By Lemma~\ref{lem:product-trace-isotopy}, the product-framed local surgery
trace contains a compact embedded \(T^2\times[0,1]\) from \(R_{a'}\) to
\(R_a\), disjoint from a collar of the local outer boundary.  The slices of
this trace give an ambient isotopy fixed near the local outer boundary, and
therefore a diffeomorphism of drilled exteriors.  By
Lemma~\ref{lem:product-trace-basis}, the induced map on the internal basis is
\(
(\alpha_{a'},\beta,\delta_{a'})
\mapsto
(\alpha_a,\beta,\delta_a).
\)
\end{proof}

\section{Product-framed marked relative handle charts}\label{sec:handle-charts}

Theorem \ref{thm:local-commutator} from Section~\ref{sec:local-commutator} applies only to
product-framed local handles.  This section proves that the fixed marked
exterior
\[
M=S^4\setminus\nu\Sigma_g^0,
\qquad
Y=\partial M=\Sigma_g\times S^1_\beta,
\]
supplies such handles.  The boundary \(Y\) is kept with its fixed product
marking throughout: all isotopies and handle slides are carried out inside
this fixed \(Y\), and no boundary reparametrization is used.

\subsection{Product lifts, endpoint flattening, and balanced bands}

\begin{definition}[product lift, \(\beta\)-degree, and product framing]
\label{def:product-lift}
Let \(C\subset\Sigma_g\) be an oriented embedded curve.  A product lift of
\(C\) to
\(
Y=\Sigma_g\times S^1_\beta
\)
is an embedded graph
\[
\widetilde C=\{(p,\ell_C(p))\mid p\in C\},
\]
where \(\ell_C:C\to S^1_\beta\) is a smooth level function.  Let
\(u_\beta\in H^1(S^1_\beta;\mathbb Z)\) be the positive generator.  Define
\[
\deg_\beta(\widetilde C)
=
\langle \ell_C^*u_\beta,[C]\rangle\in\mathbb Z.
\]
We say that \(\widetilde C\) has zero \(\beta\)-winding if
\(\deg_\beta(\widetilde C)=0\).

A product annular framing of \(\widetilde C\) is defined by choosing an
oriented annular neighborhood
\(
N_\Sigma(C)\cong C\times[-\epsilon,\epsilon]
\)
in \(\Sigma_g\), extending \(\ell_C\) constantly in the normal coordinate,
\[
\overline\ell_C(p,s)=\ell_C(p),
\]
and taking the graph over a surface-normal pushoff \(C^+\subset N_\Sigma(C)\).
For a constant-level lift \(C\times\{\theta\}\), this is the usual surface
pushoff \(C^+\times\{\theta\}\).

If \(I\subset C\) is a compact arc, we say that \(\widetilde C\) is
endpoint-flat on \(I\) if \(\ell_C|_I\) is constant and the extended level
function on the product annulus over \(I\) is constant in both the arc and
normal directions.
\end{definition}

In what follows, the same product-framing convention appears in two equivalent
forms.  For a relative \(2\)-handle attached along a constant-level curve
\(C\times\{\theta\}\subset\Sigma_g\times S^1_\beta\), the product framing is
the handle framing whose attaching longitude is the surface pushoff
\(C^+\times\{\theta\}\).  For a graph lift
\[
\widetilde C=\{(p,\ell_C(p))\mid p\in C\},
\]
the corresponding product annular framing is obtained by extending
\(\ell_C\) constantly in the surface-normal direction.  Thus product-framed
surgery, product-framed relative handles, and product annular framings all
refer to the same choice of surface-pushoff longitude in
\(\Sigma_g\times S^1_\beta\).

\begin{lemma}
\label{lem:zero-winding-to-constant}
Let \(\widetilde C\subset\Sigma_g\times S^1_\beta\) be a zero-winding product
lift of \(C\), with product annular framing.  Then, for every
\(\theta\in S^1_\beta\), \(\widetilde C\) is framed-isotopic in the fixed
boundary \(Y\) to
\(
C\times\{\theta\}
\)
with product framing.
\end{lemma}

\begin{proof}
Since \(\deg_\beta(\widetilde C)=0\), the level function \(\ell_C\) has a lift to the universal cover $\mathbb{R}$.  Choose lifts
\(\widetilde\ell_C:C\to\mathbb R\) and \(\widetilde\theta\in\mathbb R\).  The
straight-line homotopy
\[
\ell_{C,t}=(1-t)\widetilde\ell_C+t\widetilde\theta\pmod{\mathbb Z}
\]
gives an isotopy through graphs.  Extending \(\ell_{C,t}\) constantly in the
normal coordinate of \(N_\Sigma(C)\) carries the product annular framing
throughout the isotopy.
\end{proof}

\begin{lemma}
\label{lem:endpoint-flattening}
Let \(L\subset Y\) be a finite product-lifted framed link, and let
\(\widetilde C\subset L\) be a zero-winding product lift of
\(C\subset\Sigma_g\) with product annular framing.  Let \(I_0\subset C\) be a
small compact arc.  If \(I_0\) is chosen sufficiently small, there is a framed
isotopy of \(\widetilde C\), supported in a product box over a slightly larger
arc \(I_1\subset C\), after which \(\widetilde C\) is endpoint-flat on
\(I_0\).  The isotopy preserves zero \(\beta\)-winding, preserves the product
annular framing, does not move or cross the other components of \(L\), and
does not change the projected curve
\(C\), its orientation, or its label.
\end{lemma}

\begin{proof}
Choose
\(
I_0\subset\operatorname{int}I_1\subset C
\)
so small that a product box over \(I_1\) is disjoint from
\(L\setminus\widetilde C\).  Since \(I_1\) is contractible,
\(\ell_C|_{I_1}\) has a real lift.  Choose a homotopy of this lifted function,
fixed near \(\partial I_1\), to a function which is constant on \(I_0\).  After
reducing modulo \(\mathbb Z\) and extending by the identity on
\(C\setminus I_1\), taking graphs gives an isotopy supported in the chosen
product box.

Extend the homotopy over the product annular framing by making it constant in
the surface-normal coordinate.  Thus the product annular framing is carried to
the product annular framing of the new graph.  The degree of \(\ell_C:C\to
S^1_\beta\) is unchanged because the level function has only been homotoped.
The projection to \(C\) is unchanged, so the projected curve, its orientation,
and its label are unchanged.  The support was chosen disjoint from all other
components.
\end{proof}

\begin{definition}[balanced collar data and balanced lifted band]
\label{def:balanced-lifted-band}
Let
\[
B=[0,1]\times[-1,1]
\]
be a rectangular surface band.  Write
\[
e_C=\{0\}\times[-1,1],
\qquad
e_D=\{1\}\times[-1,1],
\]
and
\[
s_+=[0,1]\times\{1\},
\qquad
s_-=[0,1]\times\{-1\}.
\]
Suppose \(e_C\) is attached to an endpoint-flat arc of \(C\), and \(e_D\) is
attached to an endpoint-flat arc of a product pushoff \(D^+\).  Let the
constant endpoint levels be \(\theta_C\) and \(\theta_D\).

Balanced collar data on a collar \(U\) of \(\partial B\) is a smooth map
\[
\Lambda_\partial:U\to S^1_\beta
\]
such that it is equal to \(\theta_C\) near \(e_C\), equal to \(\theta_D\) near
\(e_D\), and on collars of the two remaining sides, which we call the long sides, it has the form
\[
\Lambda_\partial(s,1)=\lambda(s),
\qquad
\Lambda_\partial(s,-1)=\lambda(s)
\]
for one smooth path \(\lambda:[0,1]\to S^1_\beta\), constant equal to
\(\theta_C\) near \(0\) and constant equal to \(\theta_D\) near \(1\).  The map $\Lambda_\partial$
is also required to be constant in the normal coordinate to \(\partial B\) on
this collar.  This last condition is the no-twist condition for the annular
thickening.

A balanced lifted band is the graph of a smooth extension
\[
\Lambda:B\to S^1_\beta
\]
of \(\Lambda_\partial\),
\[
\widetilde B=\{(q,\Lambda(q))\mid q\in B\}\subset Y,
\]
whose interior is disjoint from the nonparticipating attaching components of the link.
\end{definition}

\begin{lemma}[graph-lift avoidance]
\label{lem:relative-graph-lift-avoidance}
Let \(B=[0,1]\times[-1,1]\subset\Sigma_g\) be an embedded surface band and
let \(L'\subset Y\) be a finite union of product lifts of embedded curves,
disjoint from the prescribed endpoint collars of \(B\).  Suppose balanced
collar data
\(
\Lambda_\partial:U\to S^1_\beta
\)
are given on a collar \(U\) of \(\partial B\), and suppose that the graph of
\(\Lambda_\partial\) is disjoint from \(L'\). Since the balanced collar data have zero boundary degree, choose a smooth
extension
\(
\Lambda_0:B\to S^1_\beta
\)
of \(\Lambda_\partial\). Since
\(B\) is a disk, such an extension \(\Lambda_0\) exists. Indeed, when one traverses \(\partial B\), the two end collars are
constant and the two long sides carry the same path \(\lambda\) with opposite
boundary orientations.  Hence the degree of
\(\Lambda_\partial|_{\partial B}:\partial B\to S^1_\beta\) is zero.  

Assume, after shrinking the band and perturbing it rel \(U\) if necessary,
that the following admissibility conditions hold:
\begin{itemize}

\item  For each component \(K\subset L'\), the projected curve \(\pi(K)\) meets
\(B_0=B\setminus\operatorname{int}U\) in a finite disjoint union of properly
embedded arcs; there are no closed loop components. Furthermore, the projected arcs arising from the components of \(L'\) are pairwise
disjoint in \(B_0\).  Equivalently, each component of
\(\pi(L')\cap B_0\) lies in the projection of a unique component of \(L'\).
\item For each such arc
\(A\), if \(f_A:A\to S^1_\beta\) is the level function of \(K\) over \(A\),
then the relative obstruction of \(\Lambda_0|_A\) to avoiding \(f_A\) is zero:
after choosing lifts \(\widehat\Lambda_0:A\to\mathbb R\) and
\(\widehat f_A:A\to\mathbb R\), the endpoint values of
\[
\widehat\Lambda_0-\widehat f_A
\]
lie in the same component \((n,n+1)\) of \(\mathbb R\setminus\mathbb Z\).
This condition is independent of the choices of lifts.
\end{itemize}

Then \(\Lambda_\partial\) extends to a smooth map
\[
\Lambda:B\to S^1_\beta
\]
such that the graph \(\widetilde B\) has interior disjoint from \(L'\), is embedded, and agrees with the
prescribed collar data.  Moreover, on the collar \(U\) of \(\partial B\), \(\Lambda\) is constant in the inward normal coordinate to each
boundary side.  Thus the collar neighborhoods of the boundary sides of
\(\widetilde B\) are the product lifts of the corresponding collar
neighborhoods in the surface band \(B\).

\end{lemma}

\begin{proof}
Let
\(
\Gamma=\pi(L')\cap B_0.
\)
By admissibility, \(\Gamma\) is a finite disjoint union of properly embedded
arcs \(A_1,\dots,A_m\).  Each \(A_j\) lies in the projection of a unique
component \(K_j\subset L'\), with level function
\(f_j:A_j\to S^1_\beta\).  The collar graph is already disjoint from \(L'\), so
at the endpoints of each \(A_j\) the prescribed endpoint values differ from
\(f_j\).

For a fixed arc \(A_j\), choose lifts \(\widehat\Lambda_0\) and
\(\widehat f_j\).  The zero relative obstruction hypothesis says that the
endpoint values of \(\widehat\Lambda_0-\widehat f_j\) lie in one component
\((n_j,n_j+1)\) of \(\mathbb R\setminus\mathbb Z\).  Choose a smooth path
\[
d_j:A_j\to(n_j,n_j+1)
\]
with these endpoint values and homotopic rel endpoints, as a path in
\(\mathbb R\), to \(\widehat\Lambda_0-\widehat f_j\).  Set
\[
\widehat\Lambda_j=\widehat f_j+d_j.
\]
Modulo \(\mathbb Z\), this defines a path \(\Lambda_j:A_j\to S^1_\beta\) which
agrees with \(\Lambda_0\), hence with \(\Lambda_\partial\), at the endpoints
and satisfies \(\Lambda_j(q)\ne f_j(q)\) for every \(q\in A_j\).  The chosen
homotopy of difference paths gives a homotopy from \(\Lambda_0|_{A_j}\) to
\(\Lambda_j\) rel endpoints.

Since the arcs \(A_j\) are pairwise disjoint, choose disjoint strip
neighborhoods \(N_j\cong A_j\times[-\epsilon,
\epsilon]\) in \(B_0\).  Use the above rel-endpoint homotopy on the core
\(A_j\times\{0\}\) and a cutoff function in the normal coordinate to define a
map on \(N_j\) which equals \(\Lambda_j\) on the core and equals \(\Lambda_0\)
near \(\partial N_j\), rel the part meeting \(U\).  Since the strips are
pairwise disjoint, these modifications may be made independently.  After
smoothing, we obtain a map \(\Lambda\) equal to \(\Lambda_\partial\) on \(U\)
and equal to \(\Lambda_j\) on each core arc \(A_j\).  Shrinking the strips if
necessary, the inequality \(\Lambda\ne f_j\) persists on the portion of
\(\pi(K_j)\cap B\) in the strip, by openness of the inequality.

There are no other possible intersections with \(L'\): a point of the graph of
\(\Lambda\) lies on \(L'\) exactly when its base point lies in \(\Gamma\) and
its \(\beta\)-coordinate equals the corresponding forbidden level.  We have
excluded this on every arc.  The graph is embedded because projection to the
embedded rectangle \(B\) is one-to-one.

Finally, \(\Lambda=\Lambda_\partial\) on the collar \(U\) of \(\partial B\). By the definition of balanced collar data, \(\Lambda_\partial\) is constant in the inward normal coordinate to each boundary side of \(B\). Hence, near each boundary side, the lifted collar of \(\widetilde B\) is the product lift of the corresponding collar of the surface band \(B\). In particular, the construction does not introduce any additional \(S^1_\beta\)-twist along the boundary collars. 
\end{proof}

\begin{remark}
\label{rem:graph-avoidance-obstruction}
The preceding admissibility hypothesis is essential.  Without it, the
relative graph-lift avoidance statement is false.  For a projected intersection
arc \(A\), the difference between a candidate section and the forbidden level
function \(f_A\) is a path in \(S^1\setminus\{1\}\) after division by \(f_A\).
The two endpoint differences must therefore lie in the same component of
\(\mathbb R\setminus\mathbb Z\) after lifting.  If they do not, every extension
with the prescribed boundary data intersects the forbidden graph.  For a closed
projected circle component, the obstruction is the relative winding
\(\deg(\Lambda|_A)-\deg(f_A)\).  In the chord-slide application below the band
is chosen thin enough that only proper arcs occur and these relative
obstructions vanish.
\end{remark}

\begin{lemma}
\label{lem:balanced-band-existence}
Let \(Y=\Sigma_g\times S^1_\beta\), and let \(L\subset Y\) be a finite framed
attaching link whose components are zero-winding product lifts of embedded
curves in \(\Sigma_g\), with product annular framings.  Assume that the
nonparticipating sublink
\[
L'=L\setminus(\widetilde C\cup\widetilde D)
\]
is projection-generic, meaning that its projected curves have no common arcs
and meet one another, if at all, only transversely in finitely many points.

Let \(\widetilde C,\widetilde D\subset L\) be two labelled components, and
write \(C=\pi(\widetilde C)\) and \(D=\pi(\widetilde D)\).  Suppose that a
chord slide replaces \(C\) by the surface band sum
\(C'=C\#_B D\), using an embedded rectangular band
\[
B=[0,1]\times[-1,1]\subset\Sigma_g
\]
whose attaching sides meet a small arc of \(C\) and a small arc of a product
pushoff \(D^+\).  After framed isotopies of the participating components,
supported over the attaching arcs and fixing every other component, the band
admits a balanced graph lift \(\widetilde B\subset Y\) such that:
\begin{enumerate}[label=(\roman*)]
\item it agrees near the attaching sides with the flattened endpoint levels
of \(\widetilde C\) and \(\widetilde D^+\);
\item its long-side collars carry balanced collar data;
\item its interior is disjoint from \(L'\).
\end{enumerate}
\end{lemma}

\begin{proof}
First perturb the core arc of the chord-slide band rel endpoints so that it is
transverse to every projected component of \(L'\) and avoids all mutual
intersection points of those projected curves.  Choose the attaching arcs and
their product boxes disjoint from \(L'\), and apply
Lemma~\ref{lem:endpoint-flattening} to \(\widetilde C\) and
\(\widetilde D\).  Let the resulting endpoint levels be \(\theta_C\) and
\(\theta_D\).

Now take a sufficiently thin rectangular neighborhood \(B\) of the chosen
core, rel the endpoint boxes.  Every component of
\(\pi(L')\cap B_0\), where \(B_0\) is the complement of the prescribed
boundary collar, is then a properly embedded crossing arc; there are no closed
components, the crossing arcs are pairwise disjoint, and each belongs to the
projection of a unique component of \(L'\).

Choose a smooth path
\(
\lambda:[0,1]\to S^1_\beta
\)
which is constant at \(\theta_C\) near \(0\), constant at \(\theta_D\) near
\(1\), and whose graph on the two long sides avoids the finitely many
forbidden levels of \(L'\).  These inequalities are strict, so compactness and
finiteness give a positive lower bound for their circular distances from the
forbidden values.  Define the preliminary extension by
\[
\Lambda_0(s,t)=\lambda(s).
\]
Thin the rectangle further so that, along every crossing arc \(A\), the
variation of both \(\Lambda_0\) and the forbidden level function \(f_A\) is
smaller than one third of this lower bound.  After choosing real lifts, the two
endpoint values of
\(
\widehat\Lambda_0-\widehat f_A
\)
then lie in the same component of \(\mathbb R\setminus\mathbb Z\).  Thus all
relative obstructions in Lemma~\ref{lem:relative-graph-lift-avoidance} vanish.

Apply that lemma.  Its modifications are performed simultaneously in disjoint
strip neighborhoods of the crossing arcs and are fixed on the prescribed
boundary collar.  The resulting graph is the required balanced lifted band.
\end{proof}

\begin{lemma}
\label{lem:product-framed-slides}
Let
\(
\widetilde C,\widetilde D\subset Y
\)
be zero-winding product lifts of oriented embedded curves
\(C,D\subset\Sigma_g\), with product annular framings.  Let
\(B\subset\Sigma_g\) be a surface band from \(C\) to a product-framed
pushoff \(D^+\).  After endpoint flattening, let \(\widetilde B\) be a
balanced lifted band.  Then the framed handle slide of \(\widetilde C\) over
\(\widetilde D\) along \(\widetilde B\) produces a zero-winding product
lift of \(C'=C\#_B D\) with product annular framing.
\end{lemma}

\begin{proof}
Let \(I_C\subset C\) and \(I_D^+\subset D^+\) be the two endpoint arcs
removed in the surface band sum.  By endpoint flattening, their level
functions are constant, so these arcs contribute zero to the
\(\beta\)-degree.  Put
\[
C_0=\overline{C\setminus I_C},
\qquad
D_0^+=\overline{D^+\setminus I_D^+}.
\]
The \(C_0\)- and \(D_0^+\)-portions contribute zero because the original
lifts have degree zero; if the band sum uses the \(D\)-portion with the
opposite orientation, the sign of this zero contribution is reversed.  The
two long sides of the band carry the same path \(\lambda\), but occur with
opposite orientations in the boundary of the band sum.  Their contributions
cancel.  Hence
\(
\deg_\beta(\widetilde C')=0.
\)

The framing is the product annular framing.  The annulus representing the
framed handle slide is obtained by gluing the product annular framing of
\(\widetilde C\), the product annular framing of the pushoff
\(\widetilde D^+\), and the lifted annular thickening of \(B\).  The pieces
agree on the endpoint collars.  On the long-side collars the level function is
\(\lambda(s)\) and is constant in the normal coordinate to the side.  Thus
the glued annulus is exactly the graph over the surface annulus defining the
product annular framing of \(C'\), with no additional twist.
\end{proof}

\subsection{Chord slides as fixed-boundary handle slides}

\begin{lemma}
\label{lem:standard-product-framed}
Let
\[
M=S^4\setminus\nu\Sigma_g^0,
\qquad
\partial M=\Sigma_g\times S^1_\beta,
\]
where \(\Sigma_g^0\subset S^3\subset S^4\) is the standard Heegaard surface
and
\[
S^3=H_-\cup_{\Sigma_g}H_+.
\]
Choose the standard complete compressing disk systems for \(H_-\) and
\(H_+\), with oriented boundary curves
\[
u_1,\dots,u_g,
\qquad
v_1,\dots,v_g,
\]
so that, for each \(i\), the curves \(u_i\) and \(v_i\) meet
transversely in exactly one point with positive local intersection sign, while
\(u_i\) and \(v_j\) are disjoint for \(i\ne j\), and so that each pair
\(u_i,v_i\) lies in one of a collection of disjoint one-holed tori in
\(\Sigma_g\).  In particular,
\[
u_i\cdot v_j=\delta_{ij}.
\]  The fixed marked exterior \(M\) has a relative handle
decomposition from \(\partial M\) with product-framed relative \(2\)-handles
attached along
\[
u_i\times\{0\},
\qquad
v_i\times\{1/2\},
\]
followed by one relative \(3\)-handle and one relative \(4\)-handle.
\end{lemma}

\begin{proof}
Choose the normal-circle coordinate so that the ray \(\beta=0\) points from
\(\Sigma_g\) into \(H_-\), while the opposite ray \(\beta=1/2\) points into
\(H_+\).  Truncate the chosen compressing disks near their boundary on the
surface-normal disk bundle.  The disks on the \(H_-\)-side become disjoint
properly embedded core disks in \(M\) with boundary
\(u_i\times\{0\}\), and the disks on the \(H_+\)-side give core disks with
boundary \(v_i\times\{1/2\}\).  Regular neighborhoods of these disks are the
relative \(2\)-handles.

Near a boundary curve \(u_i\), the attaching longitude determined by its
compressing disk is obtained by pushing \(u_i\) in the surface-normal
direction in \(\Sigma_g\) while keeping the \(\beta\)-coordinate equal to
\(0\).  It is therefore the surface pushoff
\(u_i^+\times\{0\}\), which is the product framing.  The same argument on
the other side gives the product longitude
\(v_i^+\times\{1/2\}\).

Let the disjoint one-holed tori containing the pairs \(u_i,v_i\) be chosen so
that their complement in \(\Sigma_g\) is planar.  By
Lemma~\ref{lem:Q-surgery-solid-torus}, product-framed surgery on each pair
\(u_i\times\{0\},v_i\times\{1/2\}\) replaces the corresponding
one-holed-torus product by \(D^2\times S^1_\beta\).  Capping every boundary
component of the planar complement shows that the outgoing boundary after all
relative \(2\)-handles is \(S^2\times S^1\).  Compressing the two standard
handlebodies along the chosen disk systems reduces the equatorial surface to
the standard unknotted \(S^2\); the remaining part of the exterior is
\(S^1\times B^3\).  Its ordinary handle decomposition has one \(0\)-handle
and one \(1\)-handle, so the dual decomposition relative to
\(S^2\times S^1\) has one relative \(3\)-handle and one relative
\(4\)-handle.
\end{proof}

Let
\(
S=\Sigma_g\setminus\operatorname{int}D
\)
be a once-bordered surface.  Fix distinct points
\(
p,q\in\partial S
\)
in a chosen short boundary interval
\(
I\subset\partial S.
\)
The point \(p\) is the basepoint used to read the markings, while the
univalent end of the tail is fixed at \(q\).

We use Bene's definitions as follows \cite[Sections~2--5]{Bene2010}.  A
bordered fatgraph is a fatgraph with one boundary cycle, all of whose vertices
are at least trivalent except for one univalent vertex; the edge incident to
that vertex is the tail.  A marking is an isotopy class of embeddings of the
bordered fatgraph in \(S\) as a spine with contractible complement, with the
univalent end of the tail at \(q\).  Equivalently, the marking is the
geometric \(\pi_1(S,p)\)-marking on the oriented edges.  The boundary cycle,
read from the tail, linearly orders the oriented edges.  Bene's greedy
algorithm uses this order to construct a canonical maximal tree \(T_G\): an
edge is added while the subgraph already chosen remains a tree.  The
complement
\[
X_G=E(G)\setminus E(T_G)
\]
is the generator set of the marked bordered fatgraph.

A linear chord diagram consists of an oriented interval, called the linear
core, together with chords whose endpoints are attached to the core.  Its
initial core segment is the tail.  It is a bordered chord diagram when the
associated fatgraph has one boundary cycle, and it is marked by a marking of
that bordered fatgraph.  For a bordered chord diagram the linear core is the
greedy maximal tree and its chords are precisely the elements of the generator
set.  The objects of Bene's marked chord groupoid are the marked bordered
chord diagrams.  A chord slide moves one endpoint of a chord along an adjacent
chord to the opposite endpoint, leaving all other chord endpoints fixed.
Bene's branch-reduction construction is defined using the canonical greedy
tree \(T_G\) and preserves the complementary generator set \(X_G\).
Accordingly, below we construct the two endpoint objects directly as marked
bordered chord diagrams, rather than applying branch reduction to an
independently chosen maximal tree.

\begin{theorem}[Bene]\label{thm:bene-chord}
Every morphism in the marked chord groupoid of \(\Sigma_{g,1}\) is a finite
composition of chord slides.  In particular, any two marked bordered chord
diagrams are connected by a finite sequence of chord slides.
\end{theorem}

\begin{proof}[Reference]
This is \cite[Theorem~5.3]{Bene2010}.  Bene's theorem supplies only the finite
sequence of marked chord slides.  The embedded surface curves, the lift to the
fixed boundary \(Y\), the product framings, and the endpoint control are proved
below.
\end{proof}

\begin{lemma}
\label{lem:chord-generator-curves}
Let \(\mathcal D\) be a marked bordered chord diagram embedded in \(S\).
For every oriented chord \(C\) of \(\mathcal D\), its geometric marking may be
represented by a based loop of the form
\[
\rho_C\,C^\circ\,\rho_C^{-1},
\]
where \(C^\circ\subset\operatorname{int}S\) is an oriented embedded closed
curve and \(\rho_C\) is a path from \(p\) to \(C^\circ\).  The curves
\(C^\circ\), one for each labelled chord, may be chosen simultaneously so
that they have no common arcs and meet, if at all, only transversely in
finitely many points.  We call the resulting labelled family the complete
closed generator system of \(\mathcal D\).

If two representatives of the same marked chord-diagram object are used, their
complete closed generator systems may be chosen to be simultaneously ambiently
isotopic in \(S\), relative to a collar of \(\partial S\), with all labels and
orientations preserved.
\end{lemma}

\begin{proof}
For a marked bordered fatgraph, the geometric marking is obtained from the
arc family dual to the embedded spine, with every dual arc based at \(p\)
\cite[Section~2.1]{Bene2010}.  Let \(\gamma_C\) be the embedded based dual arc
corresponding to an oriented chord \(C\).  Choose a small half-disk
\(U\subset S\) about \(p\), contained in the fixed boundary interval collar
and meeting every \(\gamma_C\) in its two endpoint subarcs.  Replace those two
subarcs by two nearby radial subarcs and join their inner endpoints by an arc
in \(U\setminus\partial S\).  Together with
\(\gamma_C\setminus\operatorname{int}U\), this gives an embedded closed curve
\(C^\circ\subset\operatorname{int}S\).  A short path \(\rho_C\subset U\) from
\(p\) to \(C^\circ\) identifies the based loop
\(\rho_C C^\circ\rho_C^{-1}\) with the original based dual arc up to based
homotopy.  Orient \(C^\circ\) so that this loop is the geometric marking of
\(C\).

The dual arcs are pairwise disjoint away from their common endpoint.  Inside
\(U\), choose the joining arcs for the finitely many chords in general
position, retaining the boundary order of their endpoint subarcs.  After an
arbitrarily small perturbation in \(U\), the resulting closed curves have no
common arcs and only transverse double intersections.  This gives the complete
closed generator system without changing any based marking.

A marking is, by definition, an isotopy class of spine embeddings.  Regard a
marking isotopy as an isotopy of a finite embedded graph in \(S\).  First
choose the two representative embeddings so that their tails and their
regular-neighborhood models agree on a fixed collar of the interval
\(I\subset\partial S\).  Since the marking isotopy is relative to this common
tail model, it may be chosen stationary on that collar.  Choose small vertex
disks and edge strips forming compact regular neighborhoods of the two graphs.
The graph isotopy extends over these product pieces to a regular-neighborhood
isotopy fixed on the common tail collar and supported away from the remainder
of \(\partial S\).  Relative isotopy extension
\cite[Chapter~8, Theorem~1.3]{Hirsch1976} therefore gives an ambient isotopy of
\(S\), fixed near \(\partial S\), which realizes the marking isotopy.  Carry
the dual arc family, the half-disk \(U\), and the joining arcs by this ambient
isotopy.  The resulting complete closed generator systems are therefore
simultaneously ambiently isotopic with all labels and orientations preserved.
\end{proof}

\begin{lemma}[genus-one chord block]
\label{lem:genus-one-chord-block}
Let \(T\cong\Sigma_{1,1}\) be an oriented one-holed torus, and let
\(C,D\subset\operatorname{int}T\) be oriented embedded curves meeting
transversely in one point with \(C\cdot D=1\).  There is a marked genus-one
bordered chord diagram embedded in \(T\) whose two labelled closed generator
curves are literally \(C\) and \(D\).
\end{lemma}

\begin{proof}
Take the standard crossing two-chord bordered diagram.  Direct traversal of
its unique boundary cycle shows that its thickening is a one-holed torus, and
the two dual chord markings, closed as in
Lemma~\ref{lem:chord-generator-curves}, form the standard oriented pair
meeting once.  A regular neighborhood of \(C\cup D\) is a one-holed torus and
its complement in \(T\) is a boundary collar.  Hence an
orientation-preserving diffeomorphism from the standard thickening to \(T\),
product on that collar, carries the standard pair to a pair isotopic to
\(C,D\).  Compose the standard marking with this diffeomorphism and use an
isotopy supported away from \(\partial T\) to make the two closed generator
curves literally equal to \(C\) and \(D\).
\end{proof}

We shall use the following elementary concatenation.  Read the chord endpoints
of two marked bordered chord diagrams in their left-to-right core orders,
place the second endpoint word immediately after the first, and retain only
the initial tail of the first diagram.  Equivalently, place the two systems of
chords in consecutive core subintervals and join the two cores across a short
interval containing no chord endpoint.  The thickening is the boundary
connected sum of the two thickenings.  Hence the concatenation of genus \(r\)
and genus \(s\) bordered chord diagrams is a genus \(r+s\) bordered chord
diagram, and its complete closed generator system is the union of the two
systems.  Marking embeddings concatenate along boundary intervals disjoint
from all closed generator curves.

\begin{lemma}[standard endpoint diagram]\label{lem:spine-free-basis}
For the standard disk systems in Lemma~\ref{lem:standard-product-framed},
there is a marked bordered chord diagram
\(
\mathcal D_{\mathrm{std}}
\)
in \(S\) whose complete closed generator system, with its orientations and
labels, is exactly
\[
u_1,\ldots,u_g,
\qquad
v_1,\ldots,v_g.
\]
The corresponding based markings form a free basis of \(\pi_1(S,p)\).
Forgetting the basepoint paths gives the standard attaching link
\[
L_0=
\{u_i\times\{0\}\}_{i=1}^g
\cup
\{v_i\times\{1/2\}\}_{i=1}^g
\subset Y.
\]
\end{lemma}

\begin{proof}
Choose the auxiliary disk \(D\) in the planar complement of the disjoint
one-holed tori containing the pairs \(u_i,v_i\).  In the \(i\)-th one-holed
torus, apply Lemma~\ref{lem:genus-one-chord-block} to the exact oriented pair
\(u_i,v_i\).  Join these tori to the fixed boundary interval \(I\) by disjoint
bands in the planar complement.  Concatenating the resulting genus-one chord
blocks along those bands gives a linear chord diagram with \(2g\) chords.
Its thickening is the boundary connected sum of the \(g\) one-holed tori and
the planar connecting region, hence is \(S\); equivalently, the associated
fatgraph has one boundary cycle and contractible complement.  Thus it is a
genuine marked bordered chord diagram.  The concatenation takes place away
from the closed generator curves, so those curves remain literally
\(u_1,\ldots,u_g,v_1,\ldots,v_g\).  Their based markings are read from \(p\)
through the fixed boundary interval and the connecting parts of the core.
Since the chords are the complement of the greedy core, their geometric
markings form the free generator set of the marked bordered chord diagram.
The displayed \(\beta\)-levels belong only to the attaching link in \(Y\), not
to the marked chord diagram.
\end{proof}

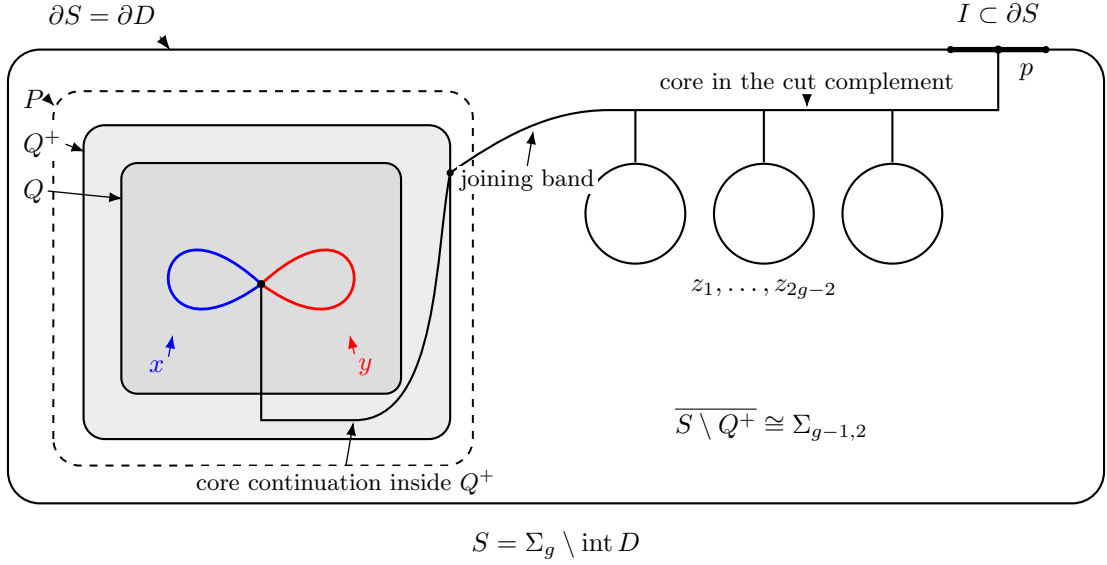
\begin{figure}
\centering
\begin{tikzpicture}[
  x=1cm,y=1cm,
  every node/.style={font=\small},
  >=Latex
]

  \draw[line width=0.8pt, rounded corners=12pt]
    (0.15,0.35) rectangle (14.65,6.35);
  \node[anchor=north] at (7.40,0.12)
    {$S=\Sigma_g\setminus\operatorname{int}D$};
  \node[anchor=south west] at (0.55,6.55) {$\partial S=\partial D$};
  \draw[-{Latex[length=2.0mm]}, line width=0.5pt]
    (2.16,6.52) -- (2.30,6.34);

  \draw[dashed, line width=0.75pt, rounded corners=10pt]
    (0.75,0.85) rectangle (6.30,5.80);
  \fill[gray!14, rounded corners=8pt]
    (1.15,1.20) rectangle (6.00,5.35);
  \draw[line width=0.8pt, rounded corners=8pt]
    (1.15,1.20) rectangle (6.00,5.35);
  \fill[gray!27, rounded corners=6pt]
    (1.65,1.80) rectangle (5.35,4.85);
  \draw[line width=0.8pt, rounded corners=6pt]
    (1.65,1.80) rectangle (5.35,4.85);

  \node[anchor=west, fill=white, inner sep=1.2pt] (labP) at (0.30,5.70) {$P$};
  \draw[-{Latex[length=1.8mm]}, line width=0.5pt]
    (labP.east) -- (0.77,5.60);
  \node[anchor=west, fill=white, inner sep=1.2pt] (labQp) at (0.30,5.10) {$Q^+$};
  \draw[-{Latex[length=1.8mm]}, line width=0.5pt]
    (labQp.east) -- (1.17,5.00);
  \node[anchor=west, fill=white, inner sep=1.2pt] (labQ) at (0.30,4.48) {$Q$};
  \draw[-{Latex[length=1.8mm]}, line width=0.5pt]
    (labQ.east) -- (1.67,4.38);

  \coordinate (v) at (3.50,3.25);
  \draw[blue, line width=1.05pt]
    (v) .. controls (1.90,4.72) and (1.82,2.02) .. (v);
  \draw[red, line width=1.05pt]
    (v) .. controls (5.10,4.72) and (5.18,2.02) .. (v);
  \fill (v) circle (1.5pt);

  \node[blue, fill=gray!27, inner sep=1.3pt] (labx) at (2.12,2.16) {$x$};
  \draw[blue, -{Latex[length=1.8mm]}, line width=0.5pt]
    (labx.north east) -- (2.33,2.58);
  \node[red, fill=gray!27, inner sep=1.3pt] (laby) at (4.88,2.16) {$y$};
  \draw[red, -{Latex[length=1.8mm]}, line width=0.5pt]
    (laby.north west) -- (4.67,2.58);

  \draw[line width=2.0pt] (12.62,6.35) -- (13.88,6.35);
  \fill (12.62,6.35) circle (1.35pt);
  \fill (13.88,6.35) circle (1.35pt);
  \coordinate (p) at (13.25,6.35);
  \fill (p) circle (1.55pt);
  \node[anchor=west] at (13.40,6.05) {$p$};
  \node[anchor=south] at (13.25,6.56) {$I\subset\partial S$};

  \coordinate (q) at (6.00,4.72);
  \draw[line width=0.8pt]
    (p) -- (13.25,5.55) -- (8.10,5.55)
    .. controls (7.30,5.55) and (6.66,5.18) .. (q);
  \fill (q) circle (1.35pt);

  \foreach \xx in {8.45,10.15,11.85} {
    \draw[line width=0.8pt] (\xx,5.55) -- (\xx,4.86);
    \draw[line width=0.8pt] (\xx,4.18) circle (0.66);
  }
  \node at (10.15,3.18) {$z_1,\ldots,z_{2g-2}$};
  \node[align=center] at (10.25,1.40)
    {$\overline{S\setminus Q^+}\cong\Sigma_{g-1,2}$};

  \node[font=\footnotesize, fill=white, inner sep=1.2pt] (labtree)
    at (10.72,5.91) {core in the cut complement};
  \draw[-{Latex[length=1.8mm]}, line width=0.5pt]
    (labtree.south) -- (10.72,5.58);
  \node[font=\footnotesize, fill=white, inner sep=1.2pt] (labjoin)
    at (7.02,4.62) {joining band};
  \draw[-{Latex[length=1.8mm]}, line width=0.5pt]
    (labjoin.north) -- (7.10,5.28);

  \draw[line width=0.8pt]
  (q) .. controls (5.80,3.72) and (5.86,1.45) .. (4.74,1.45)
    -- (3.50,1.45) -- (v);
  \node[font=\footnotesize, fill=white, inner sep=1.2pt] (labcont)
    at (4.62,0.63) {core continuation inside $Q^+$};
  \draw[-{Latex[length=1.8mm]}, line width=0.5pt]
    (labcont.north) -- (4.72,1.43);

\end{tikzpicture}
\caption{Schematic location of the target closed generator curves.  The
actual target marked bordered chord diagram is constructed by concatenating a
genus-one chord block in $Q^+$ with a genus-$(g-1)$ chord diagram in the cut
complement.  The exact curves $x$ and $y$ are the only closed generator curves
meeting $Q^+$; the remaining curves lie in
$\overline{S\setminus Q^+}\cong\Sigma_{g-1,2}$.}
\label{fig:localized-spine}
\end{figure}

\begin{lemma}\label{lem:target-spine}
Let
\[
P\cong\Sigma_{1,2}\subset\Sigma_g.
\]
Let \(x,y\subset P\) be oriented embedded simple closed curves meeting
transversely in exactly one point \(q\), with positive local intersection
sign.  In particular, \(x\cdot y=1\).  Let
\(
Q=N(x\cup y)\subset P.
\)  Choose a collar enlargement
\[
Q\subset\operatorname{int}Q^+\subset P,
\]
and assume that the auxiliary disk \(D\) used to form
\(S=\Sigma_g\setminus\operatorname{int}D\) is disjoint from \(Q^+\).  There
is a marked bordered chord diagram
\(
\mathcal D_{\mathrm{tar}}
\)
whose complete closed generator system is exactly
\[
x,y,z_1,\ldots,z_{2g-2},
\]
where \(x\) and \(y\) have the prescribed orientations, every
\(z_i\) lies in \(\overline{S\setminus Q^+}\), and \(x,y\) are the only
closed generator curves meeting \(Q^+\).  Their based markings form a free
basis of \(\pi_1(S,p)\).
\end{lemma}

\begin{proof}
The collar enlargement \(Q^+\) is a one-holed torus.  Since it has one
boundary component, the closure
\[
B=\overline{S\setminus Q^+}
\]
is connected, and
\[
\chi(B)=\chi(S)-\chi(Q^+)=(1-2g)-(-1)=2-2g.
\]
It has boundary components \(\partial S\) and \(\partial Q^+\), so
\(
B\cong\Sigma_{g-1,2}.
\)
Choose a properly embedded arc
\(
\eta\subset B
\)
joining these two boundary components.  Cutting \(B\) along \(\eta\) gives a
once-bordered surface
\(
B_\eta\cong\Sigma_{g-1,1}.
\)

Apply Lemma~\ref{lem:genus-one-chord-block} in \(Q^+\) to the literal curves
\(x,y\).  In \(B_\eta\), choose \(g-1\) disjoint one-holed tori with planar
complement and choose standard oriented pairs in them; denote their
\(2g-2\) curves by
\[
z_1,\ldots,z_{2g-2}.
\]
They lie in the interior of \(B_\eta\), hence outside \(Q^+\).  As in
Lemma~\ref{lem:spine-free-basis}, concatenate the corresponding genus-one
chord blocks to obtain a marked genus-\((g-1)\) bordered chord diagram in
\(B_\eta\).

Regluing the two sides of \(\eta\) and using a narrow neighborhood of
\(\eta\) as the connecting band realizes \(S\) as the boundary connected sum
of \(B_\eta\) and \(Q^+\).  Concatenate the two marked chord diagrams along
this band, with the tail fixed at the chosen interval \(I\subset\partial S\).
The associated fatgraph has one boundary cycle, and its complement is obtained
by joining the two complement disks across the connecting band; it is again a
disk.  Hence the result is a genuine marked bordered chord diagram in Bene's
category.  The connecting band is disjoint from all the closed generator
curves.  Consequently its complete closed generator system is literally
\(x,y,z_1,\ldots,z_{2g-2}\), with only \(x,y\) meeting \(Q^+\).  The marking
embedding supplies the chosen paths from \(p\), and the chord markings form a
free basis because the chords are the generator set complementary to the
greedy core.  Figure~\ref{fig:localized-spine} records only this location
information.
\end{proof}

\begin{lemma}[geometric realization of a chord slide]
\label{lem:geometric-chord-slide}
Let \(\mathcal D\to\mathcal D'\) be a chord slide, and let
\(\mathcal C(\mathcal D)\) be a complete closed generator system.  Write
\(C\) for the labelled generator which is replaced and \(D\) for the other
participating generator.  Then the systems may be chosen so that the
transition from \(\mathcal C(\mathcal D)\) to
\(\mathcal C(\mathcal D')\) has the following form.

There is an embedded rectangular surface band \(B\) from a small arc of
\(C\) to a small arc of a surface pushoff \(D^+\), disjoint from the other
closed generator curves near its attaching sides, such that the new curve is
the oriented surface band sum
\[
C'=C\#_B D.
\]
Every nonparticipating closed generator curve is unchanged.  With the
orientations and basepoint paths determined by the marked chord diagrams, the
six possible new based generators are exactly
\[
DC,\qquad D^{-1}C,\qquad CD,\qquad CD^{-1},\qquad
DC^{-1},\qquad C^{-1}D.
\]
The band may be chosen arbitrarily thin away from its attaching sides and
transverse to every nonparticipating curve; in particular, it has the local
product behavior and projection-genericity required in
Lemma~\ref{lem:balanced-band-existence}.

A cyclic change in the order in which a subsequence of chord generators is
read changes only the ordered marking data.  It does not move any closed
generator curve or change the corresponding framed relative \(2\)-handle.
\end{lemma}

\begin{proof}
Choose the dual based arcs used in
Lemma~\ref{lem:chord-generator-curves} in a small oriented regular
neighborhood of the marked chord diagram.  A chord slide is the composition
of the two Whitehead moves described in
\cite[Observation~4.2]{Bene2010}.  In the dual ideal triangulation, a
Whitehead move flips the arc dual to the collapsed edge.  In the disk dual to
the two local moves, all generator arcs except the removed one are fixed, and
the new generator arc is obtained from the removed arc by sliding one of its
ends along the other participating generator arc.  Equivalently, it is the
appropriate boundary arc of a thin regular neighborhood of the two old arcs.
After closing the based arcs as in
Lemma~\ref{lem:chord-generator-curves}, this regular-neighborhood strip is an
embedded rectangular surface band from a small arc of the curve labelled
\(C\) to a small arc of a surface pushoff of the curve labelled \(D\).
Replacing the attaching arcs by its two long sides is precisely the stated
surface band sum.  Outside this disk the dual arc family, and hence every
nonparticipating closed generator curve, is unchanged.

Bene computes the marking change in the six local configurations in
\cite[Figure~4.2]{Bene2010}.  In Bene's notation the new generator replacing
\(d\) is, in case order,
\[
cd,\quad c^{-1}d,\quad dc,\quad dc^{-1},\quad
cd^{-1},\quad d^{-1}c.
\]
Renaming the removed generator \(d\) as \(C\) and the other participating
generator \(c\) as \(D\) gives exactly the six words displayed in the
statement.  Geometrically, attaching the band on the side for
which the based traversal follows the \(D\)-portion before the
\(C\)-portion gives a left product, and attaching on the other side gives a
right product.  Traversing the \(D\)-portion with the opposite orientation
produces \(D^{-1}\).  The remaining two possibilities use the same underlying
bands with the opposite orientation on the resulting curve, since
\[
DC^{-1}=(CD^{-1})^{-1},
\qquad
C^{-1}D=(D^{-1}C)^{-1}.
\]
This verifies both the free-group word and the embedded oriented surface move
in every signed case.

The slide neighborhood contains no endpoint of a nonparticipating chord.
Choose its attaching sides in product neighborhoods disjoint from the other
closed generator curves.  Its core may then be perturbed rel endpoints to be
transverse to those curves and to avoid their mutual intersection points; a
sufficiently thin rectangular neighborhood has exactly the local form assumed
in Lemma~\ref{lem:balanced-band-existence}.  Finally, the cyclic permutations
noted by Bene in cases three through six reorder the linearly read generator
set but do not alter the underlying chords away from the one slide.  The
attaching link is unordered, so this bookkeeping change has no geometric or
framing effect.
\end{proof}

\begin{lemma}
\label{lem:coherent-chord-sequence}
Let
\[
\mathcal D_0\longrightarrow\mathcal D_1\longrightarrow\cdots
\longrightarrow\mathcal D_N
\]
be a finite sequence of marked chord slides.  Given complete closed generator
systems at the two endpoints, the intermediate representatives can be chosen
to give a coherent sequence of complete embedded labelled curve systems
\[
\mathcal C_0,\mathcal C_1,\ldots,\mathcal C_N
\]
in the fixed surface \(S\) such that each transition is the embedded oriented
band sum of Lemma~\ref{lem:geometric-chord-slide}, followed only by a surface
ambient isotopy fixed near \(\partial S\).  Every nonparticipating labelled
curve is fixed during the band sum.

If \(\mathcal D_N\) is the target marked chord-diagram object of
Lemma~\ref{lem:target-spine}, then \(\mathcal C_N\) is simultaneously
ambiently isotopic through an isotopy fixed near \(\partial S\), as a
complete labelled oriented system, to the literal target system
\[
x,y,z_1,\ldots,z_{2g-2}.
\]
\end{lemma}

\begin{proof}
Start with a representative of \(\mathcal D_0\) and its chosen complete
closed generator system.  For the first chord slide, perform the local band
sum from Lemma~\ref{lem:geometric-chord-slide}.  The proof of that lemma
constructs, in the same thickened chord-diagram neighborhood, a representative
of \(\mathcal D_1\) whose complete closed generator system is the band-summed
system, up to an isotopy supported in that neighborhood and fixed on all
nonparticipating curves.  Apply this isotopy, and continue inductively.  Thus
the label and orientation of every curve agree at the output of one step and
the input of the next.  The pushoff and thin-band choices in
Lemma~\ref{lem:geometric-chord-slide} may be made at each stage so that the
new complete system has no common arcs and all its intersections are
transverse; the next band is then chosen for that representative.

At the terminal stage, the final marked chord-diagram object is
\(\mathcal D_N\).  Equality of marked chord-diagram objects means that the
underlying tail-preserving bordered-fatgraph isomorphism preserves the marking,
which is an isotopy class of spine embeddings.  Since a bordered fatgraph has
no nontrivial automorphism fixing its tail \cite[Section~2]{Bene2010}, this
identification preserves every chord label.  Carrying the dual arc family
through the marking isotopy gives a simultaneous isotopy from the terminal
closed generator system to the chosen system of the target representative.
By Lemma~\ref{lem:chord-generator-curves}, it extends to an ambient isotopy of
\(S\) fixed on the common tail collar and hence fixed near \(\partial S\).  For the target object of
Lemma~\ref{lem:target-spine}, the chosen system is literally
\(x,y,z_1,\ldots,z_{2g-2}\).
\end{proof}

\begin{lemma}
\label{lem:controlled-chord-slide-lift}
Let \(L\subset Y\) be a current attaching link whose labelled components are
zero-winding product lifts with product annular framings, and suppose that the
full framed link is embedded and projection-generic.  Suppose a Bene chord
slide replaces the labelled closed generator curve \(C\) by the embedded
oriented surface band sum
\(
C'=C\#_B D
\)
supplied by Lemma~\ref{lem:geometric-chord-slide}.  Then the chord slide is
realized by a relative \(2\)-handle slide inside the fixed boundary \(Y\) such
that:
\begin{enumerate}[label=(\roman*)]
\item the nonparticipating components are fixed;
\item the new attaching component is a zero-winding product lift of \(C'\);
\item the new attaching component has product annular framing;
\item the full framed attaching link remains embedded;
\item no boundary reparametrization of \(Y\) is used.
\end{enumerate}
\end{lemma}

\begin{proof}
By Lemma~\ref{lem:geometric-chord-slide}, the attaching sides of \(B\) lie in
small product neighborhoods disjoint from every nonparticipating projected
curve, and the band can be chosen thin and projection-generic.  The surface
band need not be disjoint from all nonparticipating projected curves; it is
chosen transverse to them, and Lemma~\ref{lem:balanced-band-existence} makes
its graph lift disjoint from the nonparticipating attaching link.  Apply
Lemma~\ref{lem:endpoint-flattening} to the participating components.  The
nonparticipating sublink is projection-generic, so
Lemma~\ref{lem:balanced-band-existence} supplies a balanced lifted band whose
interior is disjoint from it.  Slide the relative \(2\)-handle attached along
\(\widetilde C\) over the handle attached along \(\widetilde D\) using this
band.  This is a framed handle slide in the fixed boundary.  The band is
disjoint from the nonparticipating link, so the resulting full link is
embedded.  Lemma~\ref{lem:product-framed-slides} gives a zero-winding product
lift of \(C'\) with product annular framing.  No boundary reparametrization is
used.
\end{proof}

\begin{proposition}[attaching-link slide theorem]
\label{prop:marked-attaching-slide}
Let
\[
P\cong\Sigma_{1,2}\subset\Sigma_g.
\]
Let \(x,y\subset P\) be oriented embedded simple closed curves meeting
transversely in exactly one point \(q\), with positive local intersection
sign.  In particular, \(x\cdot y=1\).  Let \(Q=N(x\cup y)\), and choose
\(
Q\subset\operatorname{int}Q^+\subset P.
\)
Then the standard relative attaching link
\[
L_0=
\{u_i\times\{0\}\}_{i=1}^g
\cup
\{v_i\times\{1/2\}\}_{i=1}^g
\subset
Y=\Sigma_g\times S^1_\beta
\]
can be changed by relative \(2\)-handle slides and framed isotopies inside the
fixed \(Y\) to an attaching link containing the exact constant-level
product-framed components
\[
x\times\{0\},
\qquad
y\times\{1/2\}.
\]
Every other attaching component projects outside \(Q^+\).  Throughout the
construction every component is a zero-winding product lift with product
annular framing, and the full framed attaching link is embedded.
\end{proposition}

\begin{proof}
Choose the auxiliary disk \(D\) in \(\Sigma_g\setminus P\).  The standard
disk systems in Lemma~\ref{lem:standard-product-framed} may be chosen so that
\(D\) lies in the planar complement of their one-holed tori.  Let
\(\mathcal D_{\mathrm{std}}\) be the standard endpoint diagram of
Lemma~\ref{lem:spine-free-basis}, and let
\(\mathcal D_{\mathrm{tar}}\) be the target endpoint diagram of
Lemma~\ref{lem:target-spine}.  They are genuine objects of the same marked
chord groupoid, with the tail fixed at the chosen boundary interval.  By
Theorem~\ref{thm:bene-chord}, there is a finite chord-slide sequence
\[
\mathcal D_{\mathrm{std}}=\mathcal D_0
\longrightarrow\mathcal D_1
\longrightarrow\cdots\longrightarrow
\mathcal D_N=\mathcal D_{\mathrm{tar}}.
\]
Lemma~\ref{lem:coherent-chord-sequence} realizes this as a coherent sequence
of complete embedded labelled surface-curve systems.  Each transition is the
embedded oriented band sum of
Lemma~\ref{lem:geometric-chord-slide}; a cyclic change in the order of the
marked generators is retained only as bookkeeping and makes no change to the
attaching link.

The initial attaching link is the constant-level lift of the complete closed
generator system of \(\mathcal D_{\mathrm{std}}\), hence is zero-winding and
product-framed.  Lift the coherent sequence inductively.  For every surface
band sum, Lemma~\ref{lem:controlled-chord-slide-lift} gives the corresponding
relative \(2\)-handle slide in the fixed \(Y\).  By
Lemma~\ref{lem:coherent-chord-sequence}, every inserted surface ambient
isotopy is fixed near \(\partial S=\partial D\).  Extend it by the identity
over \(D\), and lift the resulting isotopy of \(\Sigma_g\) by
\[
H_t\times\operatorname{id}_{S^1_\beta}:Y\longrightarrow Y.
\]
Applied to a graph lift, this carries its level function by the same
parametrization, so its \(\beta\)-degree is unchanged; applied to its product
annular framing, it carries the surface pushoff to the surface pushoff.  Thus
these lifted isotopies preserve zero \(\beta\)-winding, product annular
framing, and embeddedness of the complete link.  The identity extension over
\(D\) makes each lift a product isotopy of the fixed boundary \(Y\), used
only as a framed isotopy of the attaching link and not as a boundary
reparametrization.

The geometric realization lemma chooses each pushoff and thin band so that the
new projected complete system is already projection-generic.  Consequently
the hypotheses of the controlled-lift lemma hold at every stage.  Induction
therefore produces a framed attaching link whose projected complete labelled
system is the terminal system of the chord sequence.

By the terminal assertion of
Lemma~\ref{lem:coherent-chord-sequence}, there is a simultaneous ambient
isotopy of \(S\), fixed near \(\partial S\), from that system to the literal
target system
\[
x,y,z_1,\ldots,z_{2g-2}.
\]
Extend this terminal isotopy by the identity over \(D\), and lift it by its
product with \(\operatorname{id}_{S^1_\beta}\).  The same argument shows that the full
attaching link remains embedded and that every component remains a
zero-winding product lift with product annular framing.  At its end, the two
distinguished projected curves are literally \(x,y\), while every other
projected component lies outside \(Q^+\).

It remains to put the two distinguished components at the required levels
simultaneously.  Let their degree-zero level functions be
\(
\ell_x:x\to S^1_\beta
\)
and
\(
\ell_y:y\to S^1_\beta.
\)
At the unique intersection point \(q\), choose real lifts \(L_x,L_y\).  Since the two graph
lifts are disjoint,
\(
\ell_x(q)\ne\ell_y(q).
\)
The ordered configuration space
\[
\{(\theta_1,\theta_2)\in S^1_\beta\times S^1_\beta
\mid \theta_1\ne\theta_2\}
\]
is connected: subtracting the first coordinate identifies it with
\(S^1_\beta\times(0,1)\).  Choose a path
\(
(\gamma_x(t),\gamma_y(t))
\)
from \((\ell_x(q),\ell_y(q))\) to \((0,1/2)\), and choose real lifts
\(A_x(t),A_y(t)\) beginning at \(L_x(q),L_y(q)\).  Define
\[
L_{x,t}(s)=A_x(t)+(1-t)(L_x(s)-L_x(q)),
\]
\[
L_{y,t}(s)=A_y(t)+(1-t)(L_y(s)-L_y(q)).
\]
Modulo \(\mathbb Z\), these are global homotopies through degree-zero level
functions, ending at the constant levels \(0\) and \(1/2\).  At \(q\) their
values follow the chosen path and remain distinct.  Away from \(q\), the base
curves \(x\) and \(y\) are disjoint.  Extending the homotopies constantly in
the surface-normal directions carries the product annular framings.  The path
is compact and avoids the diagonal, so the two product annuli can be chosen
sufficiently small to remain disjoint throughout.  The isotopy is supported in
\(Q^+\times S^1_\beta\), while every other component projects outside
\(Q^+\).  It therefore fixes the other components and produces exactly
\[
x\times\{0\},
\qquad
y\times\{1/2\}
\]
with product framings.
\end{proof}

\begin{proposition}[marked product-framed relative handle chart]
\label{prop:product-framed-chart}
Let
\[
P\cong\Sigma_{1,2}\subset\Sigma_g.
\]
Let \(x,y\subset P\) be oriented embedded simple closed curves meeting
transversely in exactly one point \(q\), with positive local intersection
sign.  In particular, \(x\cdot y=1\).
The fixed marked exterior
\[
M=S^4\setminus\nu\Sigma_g^0,
\qquad
Y=\partial M=\Sigma_g\times S^1_\beta,
\]
has a relative handle decomposition from the fixed boundary in which:
\begin{enumerate}[label=(\roman*)]
\item the first two relative \(2\)-handles are attached exactly along
\[
x\times\{0\},
\qquad
y\times\{1/2\},
\]
with product framings;
\item the remaining relative \(2\)-handles are attached afterward;
\item the relative \(3\)- and \(4\)-handles are attached last;
\item the decomposition is obtained from the standard one by relative
\(2\)-handle slides, framed isotopies inside the fixed boundary \(Y\), and
reordering of the relative \(2\)-handles.
\end{enumerate}
\end{proposition}

\begin{proof}
Start with the relative handle decomposition of
Lemma~\ref{lem:standard-product-framed}.  Apply
Proposition~\ref{prop:marked-attaching-slide} to its framed attaching link.
Relative \(2\)-handle slides and framed isotopies inside the fixed \(Y\) do
not change the relative diffeomorphism type or the boundary marking.

The resulting attaching link is a disjoint union of the two distinguished
components and the remaining framed sublink.  Choose pairwise disjoint framed
tubular neighborhoods of all components.  Attach first the handles along
\(x\times\{0\}\) and \(y\times\{1/2\}\).  Every other attaching circle and
its framing annulus lies in the complement of those two tubular neighborhoods,
so it persists canonically in the complement portion of the new outgoing
boundary.  Attach the remaining relative \(2\)-handles there with their
original framings.  Gluing the disjoint handle neighborhoods in this order or
simultaneously gives canonically diffeomorphic relative handlebodies rel the
incoming boundary; no additional handle slide or boundary reparametrization
is involved.  Finally attach the relative \(3\)- and \(4\)-handles with their
attaching data transported through the preceding handle-slide identification.
\end{proof}

\begin{lemma}[globalization of the local move]
\label{lem:local-sector-globalization}
Let
\[
P\cong\Sigma_{1,2}\subset\Sigma_g,
\qquad
\partial P=a\cup(-a'),
\]
and let \(x,y\subset P\) be oriented embedded simple closed curves meeting
transversely in exactly one point \(q\), with positive local intersection
sign.  Suppose, in the sense of
Convention~\ref{conv:local-boundary-word}, that
\(
a'=[x,y]a.
\)  After applying
Proposition~\ref{prop:product-framed-chart}, the local commutator
diffeomorphism extends to a global diffeomorphism
\[
A_{a'}\longrightarrow A_a
\]
which is rel \(\partial M=\Sigma_g\times S^1_\beta\) and carries
\[
(\alpha_{a'},\beta,\delta_{a'})
\longmapsto
(\alpha_a,\beta,\delta_a).
\]
\end{lemma}

\begin{proof}
Let \(X_2\) be the relative handle stage obtained from \(Y\times[0,1]\) by
attaching only the first two handles in
Proposition~\ref{prop:product-framed-chart}, namely the product-framed handles
\(h_x,h_y\) along \(x\times\{0\}\) and \(y\times\{1/2\}\).  The product
inclusion of \(P\times S^1_\beta\times[0,1]\), together with these two handle
neighborhoods, identifies the local commutator model
\(W_{\mathrm{loc}}\) with a submanifold of \(X_2\).

Choose \(R_{a'}\) and \(R_a\) at the prescribed interior collar level.  In
the construction of Lemma~\ref{lem:product-trace-isotopy}, push the
post-surgery copy of \(A\times S^1_\beta\), its product ramps, and a small
regular neighborhood of the resulting trace away from collars of both the
incoming boundary \(Y\times\{0\}\) and the entire outgoing boundary of
\(X_2\).  Use the framed tubular-neighborhood embeddings constructed in
Lemma~\ref{lem:product-trace-basis}.  Isotopy extension for those embeddings
\cite[Chapter~8, Theorem~1.3]{Hirsch1976} then gives a diffeomorphism
\[
X_2\setminus\nu R_{a'}
\longrightarrow
X_2\setminus\nu R_a
\]
which is the identity on collars of both boundary components and carries
\[
(\alpha_{a'},\beta,\delta_{a'})
\longmapsto
(\alpha_a,\beta,\delta_a)
\]
exactly.

Every remaining relative \(2\)-handle attaching circle and framing annulus
lies in the outgoing boundary of \(X_2\).  Since the diffeomorphism is
literally the identity on a collar of that entire boundary, extend it by the
identity over each remaining relative \(2\)-handle.  The extended map is again
the identity near the outgoing boundary of the complete relative
\(2\)-handle stage, and therefore extends by the identity over the relative
\(3\)- and \(4\)-handles.  It is rel the original incoming boundary
\(Y=\partial M\).  Removing the chosen rim-torus neighborhoods gives the
claimed diffeomorphism \(A_{a'}\to A_a\).  No spatial separation from the
later attaching circles is required.
\end{proof}

\section{The marked homology-relative rim lemma}\label{sec:homology-relative}

We now globalize the local commutator move to rim tori determined by
homologous curves.

\begin{proposition}[marked homology-relative rim lemma]\label{prop:homology-relative}
Let \(g\ge3\).  If \(a,a'\subset\Sigma_g\) are oriented nonseparating curves
with
\(
[a]=[a']\in H_1(\Sigma_g;\mathbb Z),
\)
then
\[
A_{a'}\cong A_a
\quad\mathrm{rel}\ \partial M,
\]
and the diffeomorphism carries
\[
(\alpha_{a'},\beta,\delta_{a'})
\mapsto
(\alpha_a,\beta,\delta_a).
\]
\end{proposition}

\begin{proof}
Let
\(
x=[a]\in H_1(\Sigma_g;\mathbb Z).
\)
By \cite[Theorem~2]{HatcherMargalit2012}, attributed there to Putman, the
complex \(C_x(\Sigma_g)\) is connected for \(g\ge3\).  Its vertices are
curves that can be oriented to represent the fixed nonzero primitive class
\(x\); this orientation is unique.  Hence there is a sequence
\[
a=a_0,a_1,\dots,a_N=a'
\]
of oriented curves representing \(x\), with each adjacent pair disjoint.

Fix \(i\).  If \(a_i\) and \(a_{i+1}\) are isotopic, choose the annulus whose
oriented boundary is \(a_i\sqcup(-a_{i+1})\).  The product isotopy through
this annulus gives
\[
\Psi_i:A_{a_{i+1}}\longrightarrow A_{a_i}
\quad\mathrm{rel}\ \partial M
\]
and preserves \((\alpha,\beta,\delta)\) in the product tubular coordinates.

Suppose instead that the two curves are nonisotopic.  They form a bounding
pair.  Choose the connected component \(P_i\) whose induced oriented boundary
is
\[
\partial P_i=a_i\sqcup(-a_{i+1}).
\]
Then \(P_i\cong\Sigma_{r_i,2}\) with \(r_i>0\).  Choose nested oriented
curves
\[
c_{i,0},c_{i,1},\dots,c_{i,r_i}
\]
with \(c_{i,0}=a_i\), \(c_{i,r_i}=a_{i+1}\), such that
\(c_{i,j-1}\sqcup(-c_{i,j})\) bounds a genus-one subsurface
\(P_{i,j}\cong\Sigma_{1,2}\).  For each \(P_{i,j}\), choose a basepoint and the paths to the curves as in
Convention~\ref{conv:local-boundary-word}.  Choose oriented embedded simple
closed curves \(x_{i,j},y_{i,j}\subset P_{i,j}\) meeting transversely in
exactly one point with positive local intersection sign.  In particular,
\[
x_{i,j}\cdot y_{i,j}=1.
\]
With the chosen paths and boundary orientations, the literal based
relation is
\[
c_{i,j}=[x_{i,j},y_{i,j}]\,c_{i,j-1}.
\]
This is obtained by taking \(x_{i,j},y_{i,j}\) as the standard oriented
curves in the one-holed-torus part and using the complementary pair of pants
to the two boundary curves.

Lemma~\ref{lem:local-sector-globalization} gives, for every \(j\), a
diffeomorphism
\[
\Phi_{i,j}:A_{c_{i,j}}\longrightarrow A_{c_{i,j-1}}
\quad\mathrm{rel}\ \partial M
\]
carrying
\[
(\alpha_{c_{i,j}},\beta,\delta_{c_{i,j}})
\longmapsto
(\alpha_{c_{i,j-1}},\beta,\delta_{c_{i,j-1}}).
\]
Define
\[
\Psi_i=\Phi_{i,1}\circ\Phi_{i,2}\circ\cdots\circ\Phi_{i,r_i}.
\]
The rightmost map is applied first, so
\[
A_{c_{i,r_i}}
\xrightarrow{\Phi_{i,r_i}}
A_{c_{i,r_i-1}}
\longrightarrow\cdots\longrightarrow
A_{c_{i,1}}
\xrightarrow{\Phi_{i,1}}
A_{c_{i,0}}.
\]
Thus \(\Psi_i:A_{a_{i+1}}\to A_{a_i}\) is rel \(\partial M\) and preserves
the ordered internal basis exactly.

Finally set
\[
\Psi=\Psi_0\circ\Psi_1\circ\cdots\circ\Psi_{N-1}.
\]
Again the rightmost map is applied first, giving a map
\(A_{a'}\to A_a\).  Every elementary map is rel the outer boundary and
carries the three fixed tubular coordinates literally to the corresponding
three coordinates.  Their composites therefore introduce no meridional shear
and satisfy
\[
(\alpha_{a'},\beta,\delta_{a'})
\longmapsto
(\alpha_a,\beta,\delta_a).
\]
\end{proof}

\begin{remark}
The above proof uses the connectivity theorem for \(g\ge3\). The genus 2 case
requires a different complex and is not included here.
\end{remark}

\begin{lemma}[orientation reversal of the internal rim basis]
\label{lem:rim-basis-orientation-reversal}
Let \(b\subset\Sigma_g\) be an oriented nonseparating curve, and let
\(\bar b\) denote the same embedded curve with the opposite orientation.
Then
\(
A_{\bar b}=A_b
\)
as drilled exteriors, and the internal boundary bases satisfy
\[
\alpha_{\bar b}=\alpha_b^{-1},
\qquad
\beta_{\bar b}=\beta_b=\beta,
\qquad
\delta_{\bar b}=\delta_b^{-1}.
\]
\end{lemma}

\begin{proof}
The drilled exterior depends on the unoriented embedded rim torus
\(
R_b=b\times S^1_\beta\subset M,
\)
so \(A_{\bar b}=A_b\) as manifolds rel outer boundary.  The generator
\(\alpha_b\) is the \(b\)-direction on the rim torus, hence reversing the
orientation of \(b\) gives
\(
\alpha_{\bar b}=\alpha_b^{-1}.
\)
The generator \(\beta\) is the oriented surface meridian \(\mu_\Sigma\), so it
is independent of the orientation of \(b\).

It remains to check \(\delta\).  The orientation of the rim torus
\(
R_b=b\times S^1_\beta
\)
is determined by the ordered tangent directions \((\alpha_b,\beta)\).  Reversing
the orientation of \(b\), while keeping \(\beta\) fixed, reverses the orientation
of \(R_b\).  The rim-torus meridian \(\delta_b\) is the positively oriented
boundary of the normal disk for which the oriented tangent plane of \(R_b\),
followed by the oriented normal \(2\)-plane, gives the orientation of \(M\).  Hence
reversing the orientation of \(R_b\) reverses the positive orientation of the
normal disk, and therefore
\(
\delta_{\bar b}=\delta_b^{-1}.
\)

Equivalently, this is forced by Lemma~\ref{lem:drilled-exterior}.  For every
oriented loop \(c\subset\Sigma_g\), the outer boundary map is
\(
c\longmapsto \delta_b^{\,b\cdot c}.
\)
For the opposite orientation \(\bar b\), the same geometric drilling gives the
same boundary element, but
\(
\bar b\cdot c=-(b\cdot c).
\)
Thus
\[
\delta_b^{\,b\cdot c}
=
\delta_{\bar b}^{\,\bar b\cdot c}
=
\delta_{\bar b}^{-(b\cdot c)}
\]
for all \(c\).  Choosing \(c\) with \(b\cdot c=1\) gives
\(
\delta_{\bar b}=\delta_b^{-1}.
\)
\end{proof}

\begin{proposition}[signed marked homology-relative rim lemma]
\label{prop:signed-homology-relative}
Let \(g\ge3\).  Let \(a,a'\subset\Sigma_g\) be oriented nonseparating curves
with
\(
[a']=-[a]\in H_1(\Sigma_g;\mathbb Z).
\)
Then
\[
A_{a'}\cong A_a
\quad\mathrm{rel}\ \partial M,
\]
and the diffeomorphism carries
\[
(\alpha_{a'},\beta,\delta_{a'})
\mapsto
(\alpha_a^{-1},\beta,\delta_a^{-1}).
\]
\end{proposition}

\begin{proof}
Let \(\overline{a'}\) denote the same embedded curve as \(a'\), with the
opposite orientation.  Since
\(
[a']=-[a],
\)
we have
\(
[\overline{a'}]=[a].
\)
By Proposition~\ref{prop:homology-relative}, there is a diffeomorphism
\[
\Psi:A_{\overline{a'}}\longrightarrow A_a
\quad\mathrm{rel}\ \partial M
\]
such that
\[
\Psi_*(\alpha_{\overline{a'}})=\alpha_a,
\qquad
\Psi_*(\beta)=\beta,
\qquad
\Psi_*(\delta_{\overline{a'}})=\delta_a.
\]
Since
\(
A_{\overline{a'}}=A_{a'}
\)
as drilled exteriors, Lemma~\ref{lem:rim-basis-orientation-reversal} gives
\[
\alpha_{\overline{a'}}=\alpha_{a'}^{-1},
\qquad
\beta_{\overline{a'}}=\beta,
\qquad
\delta_{\overline{a'}}=\delta_{a'}^{-1}.
\]
Therefore
\[
\Psi_*(\alpha_{a'}^{-1})=\alpha_a,
\qquad
\Psi_*(\beta)=\beta,
\qquad
\Psi_*(\delta_{a'}^{-1})=\delta_a.
\]
Equivalently,
\[
\Psi_*(\alpha_{a'})=\alpha_a^{-1},
\qquad
\Psi_*(\beta)=\beta,
\qquad
\Psi_*(\delta_{a'})=\delta_a^{-1}.
\]
This proves the proposition.  There is no additional shear term, because
Proposition~\ref{prop:homology-relative} gives an exact basis-preserving map
for the same oriented homology class, and
Lemma~\ref{lem:rim-basis-orientation-reversal} only changes the orientation
labels on the same three geometric circles.
\end{proof}

\section{The lower bound}\label{sec:lower}
We now assemble the lower bound.
\begin{lemma}\label{lem:normal-product}
Let \(f\in\Stab(q_0)\), and let
\(
F:(S^4,\Sigma_g^0)\to(S^4,\Sigma_g^0)
\)
be an ambient extension of \(f\). Then \(F\) may be isotoped through pair
diffeomorphisms so that it is product-normal near \(\Sigma_g^0\).
\end{lemma}

\begin{proof}
Choose the fixed tubular coordinates
\[
\nu\Sigma_g^0\cong\Sigma_g\times D^2,
\]
and write \(f\) also for the restriction of \(F\) to the zero section.  The
derivative of \(F\) along the zero section induces an oriented isomorphism of
the normal bundle over \(f\).  The restriction of \(F\) to a sufficiently
small tubular neighborhood and the bundle map determined by this normal
isomorphism are tubular-neighborhood embeddings with the same restriction and
the same normal derivative on the zero section.  By uniqueness of tubular
neighborhoods, after shrinking the disk fibers they are isotopic through
such embeddings, fixed on the zero section.  Isotopy extension, supported in
the original tubular neighborhood, therefore changes \(F\) through pair
diffeomorphisms as in \cite[Chapter~4, Theorem~5.3 and Chapter~8, Theorem~1.3]{Hirsch1976}, so that near the zero section it is a bundle map over \(f\).

In the fixed oriented trivialization, its fiber maps form a smooth map
\[
\Sigma_g\longrightarrow GL^+(2,\mathbb R).
\]
The positivity follows because \(F\) preserves the orientations of \(S^4\)
and \(\Sigma_g^0\), and hence the orientation of the normal plane.  The
standard deformation of \(GL^+(2,\mathbb R)\) onto \(SO(2)\) gives a
deformation through orientation-preserving fiber maps.  On a possibly smaller
tubular neighborhood these are bundle embeddings fixed on the zero section,
and isotopy extension again realizes the deformation by a supported isotopy
through pair diffeomorphisms.  We may therefore assume that
\[
(p,z)\longmapsto(f(p),r(p)z)
\]
near the surface, for a smooth map \(r:\Sigma_g\to S^1\).  After shrinking
the disk fibers once more, this formula preserves the chosen tubular
neighborhood, so \(F\) restricts to a diffeomorphism of its exterior.

On the surface-normal circle bundle
\[
\partial M=\Sigma_g\times S^1_\beta,
\]
the same formula is
\[
(p,\beta)\longmapsto(f(p),r(p)\beta).
\]
For a based loop \(c\subset\Sigma_g\), it follows that
\[
c\longmapsto f_*(c)\beta^{u(c)},
\qquad
u(c)=\deg(r|_c).
\]
Here the homomorphism \(u:H_1(\Sigma_g;\mathbb Z)\to\mathbb Z\) is the
homomorphism corresponding to the cohomology class of \(r\).  Since \(F\)
preserves the oriented normal plane, it sends the oriented meridian \(\beta\)
to itself.

Now use the standard exterior
\[
M=S^4\setminus\nu\Sigma_g^0,
\qquad
\pi_1(M)=\mathbb Z\langle\beta\rangle.
\]
Both \(c\) and \(f_*(c)\) map trivially under
\(\pi_1(\partial M)\to\pi_1(M)\).  Naturality of this boundary inclusion gives
\[
1
=
F_*(1)
=
\beta^{u(c)}
\quad\text{in }\pi_1(M)
\]
for every \(c\).  Since \(\beta\) has infinite order, \(u(c)=0\) for every
\(c\).  Hence
\[
[r]=0\in H^1(\Sigma_g;\mathbb Z).
\]

Choose a homotopy \(r_t\) from \(r\) to the constant map \(1\).  The maps
\[
(p,z)\longmapsto(f(p),r_t(p)z)
\]
form an isotopy of the restriction of \(F\) on the tubular neighborhood,
through bundle diffeomorphisms whose restriction to the zero section is
constantly \(f\).  Relative isotopy extension, supported in a slightly larger
tubular neighborhood and fixed on the zero section, extends it to an ambient
isotopy through pair diffeomorphisms.  At the end of this isotopy,
\[
(p,z)\longmapsto(f(p),z)
\]
on a possibly smaller tubular neighborhood of \(\Sigma_g^0\).  Thus \(F\)
is product-normal near \(\Sigma_g^0\), as required.
\end{proof}

\begin{proposition}[lower bound]
\label{prop:lower}
Let \(g\ge3\), and let \(J\subset S^3\) be a nontrivial knot.  Then
\[
\Stab_{\Mod(\Sigma_g)}(q_0)
\cap
\Stab_{\Mod(\Sigma_g)}(\Gamma_\mu(J)\cdot[a])
\subseteq
E(\Sigma_{g,a,J}).
\]
\end{proposition}

\begin{proof}
Let
\[
f\in
\Stab_{\Mod(\Sigma_g)}(q_0)
\cap
\Stab_{\Mod(\Sigma_g)}(\Gamma_\mu(J)\cdot[a]).
\]
Since
\(
f_*(\Gamma_\mu(J)\cdot[a])
=
\Gamma_\mu(J)\cdot[a],
\)
we have
\(
[a]\in\Gamma_\mu(J)\cdot[a].
\)
Hence
\(
f_*[a]\in\Gamma_\mu(J)\cdot[a].
\)
Thus there is a sign
\(
\varepsilon\in\Gamma_\mu(J)
\)
such that
\(
f_*[a]=\varepsilon[a].
\)

By Hirose's theorem \cite{Hirose2002}, \(f\) extends over the unknotted pair
\(
(S^4,\Sigma_g^0).
\)
After applying Lemma~\ref{lem:normal-product}, choose such an extension
\[
\widehat F_f:(S^4,\Sigma_g^0)\to(S^4,\Sigma_g^0)
\]
which is product-normal near \(\Sigma_g^0\).  We write \(f(a)\) for the
oriented image of \(a\) under the surface restriction of \(\widehat F_f\).
On the surface-normal circle bundle
\[
\partial M=\Sigma_g\times S^1_\beta,
\]
the map \(\widehat F_f\) has the form
\[
(p,\beta)\longmapsto(f(p),\beta).
\]
Consequently, it sends
\(
R_a=a\times S^1_\beta
\)
to
\(
R_{f(a)}=f(a)\times S^1_\beta.
\)
Choose the product tubular neighborhood of \(R_{f(a)}\) to be the image of
the chosen product tubular neighborhood of \(R_a\).  Then
\[
\widehat F_f|_{A_a}:A_a\to A_{f(a)}
\]
has, on the internal boundary, the literal product-coordinate restriction
\[
(\alpha_a,\beta,\delta_a)
\longmapsto
(\alpha_{f(a)},\beta,\delta_{f(a)}).
\]

We now construct a pair diffeomorphism
\[
\Phi_f:
(S^4,\Sigma_{g,a,J})\to(S^4,\Sigma_{g,f(a),J}).
\]
On the drilled-exterior piece \(A_a\), use
\(\widehat F_f|_{A_a}\), and on the inserted piece
\(S^1_s\times E(J)\), use the identity.  On the common parameterized
\(3\)-torus, the two composites with the source and target rim-surgery
gluings are equal as maps: both are the product of
\[
s\longmapsto\alpha_{f(a)},
\qquad
\mu_J\longmapsto\beta,
\qquad
\lambda_J\longmapsto\delta_{f(a)}.
\]
Thus the two piecewise maps agree literally on the common boundary and glue
to a diffeomorphism of the complements.  On
\(\partial M=\Sigma_g\times S^1_\beta\), the glued map is
\[
(p,\beta)\longmapsto(f(p),\beta),
\]
so it extends over the surface normal disk bundle by
\[
(p,z)\longmapsto(f(p),z).
\]
The resulting orientation-preserving pair diffeomorphism is \(\Phi_f\), and
it induces exactly \(f\) under the canonical markings.

Since
\(
[f(a)]=\varepsilon[a],
\)
Proposition~\ref{prop:homology-relative} applies when \(\varepsilon=1\), and
Proposition~\ref{prop:signed-homology-relative} applies when
\(\varepsilon=-1\).  Thus, in either case, there is a diffeomorphism
\[
\Psi_\varepsilon:A_{f(a)}\to A_a
\quad\mathrm{rel}\ \partial M
\]
carrying
\[
(\alpha_{f(a)},\beta,\delta_{f(a)})
\longmapsto
(\alpha_a^\varepsilon,\beta,\delta_a^\varepsilon).
\]
The construction of those propositions uses the framed tubular-neighborhood
maps of Lemma~\ref{lem:product-trace-basis}.  Hence
\(\Psi_\varepsilon\) may be chosen so that its internal-boundary restriction is
literally the product of the displayed maps of the three parameterized circle
factors.

By the definition of \(\Gamma_\mu(J)\), there is a diffeomorphism
\[
h_\varepsilon:E(J)\to E(J)
\]
such that
\[
(h_\varepsilon)_*(\mu_J)=\mu_J,
\qquad
(h_\varepsilon)_*(\lambda_J)=\lambda_J^\varepsilon.
\]
Choose on \(\partial E(J)\) the diffeomorphism carrying the parameterized
meridian and longitude circles literally to \(\mu_J\) and
\(\lambda_J^\varepsilon\).  Since a torus diffeomorphism is determined up to
isotopy by its action on \(H_1\) \cite[Theorem~2.5]{FarbMargalit2012}, with
the orientation-reversing case reduced to the orientation-preserving case by
a fixed reflection, the restriction of \(h_\varepsilon\) has this mapping
class.  Extending a boundary isotopy across a collar, isotope
\(h_\varepsilon\) so that it is this chosen diffeomorphism on the boundary and
is product with it on a smaller collar.  This does not change its induced
group automorphism or its orientation sign.

Define
\[
\Theta_\varepsilon:S^1_s\times E(J)\to S^1_s\times E(J)
\]
by
\[
\Theta_\varepsilon(s,x)=(s^\varepsilon,h_\varepsilon(x)).
\]
On the parameterized boundary, this is literally the product of
\[
s\longmapsto s^\varepsilon,
\qquad
\mu_J\longmapsto\mu_J,
\qquad
\lambda_J\longmapsto\lambda_J^\varepsilon.
\]

We now construct
\[
\Omega_\varepsilon:
(S^4,\Sigma_{g,f(a),J})\to(S^4,\Sigma_{g,a,J}).
\]
On \(A_{f(a)}\), use \(\Psi_\varepsilon\), and on
\(S^1_s\times E(J)\), use \(\Theta_\varepsilon\).  Under the source gluing
for rim surgery along \(f(a)\), followed by \(\Psi_\varepsilon\), the three
parameterized circle factors are sent literally by
\[
s\longmapsto\alpha_a^\varepsilon,
\qquad
\mu_J\longmapsto\beta,
\qquad
\lambda_J\longmapsto\delta_a^\varepsilon.
\]
Under \(\Theta_\varepsilon\), followed by the target gluing for rim surgery
along \(a\), the same three circle factors are sent by the same product map.
Hence the two restrictions are equal as diffeomorphisms of the common
parameterized \(3\)-torus, not merely equal on homology or fundamental groups.
They therefore glue to a diffeomorphism of complements.  Since
\(\Psi_\varepsilon\) is the identity on a collar of \(\partial M\), the glued
map is the identity on the outer boundary collar and extends over the surface
normal disk bundle by the identity.  The resulting correction map
\(\Omega_\varepsilon\) induces the identity under the canonical marking.

The map \(\Omega_\varepsilon\) is orientation-preserving.  The connected
oriented drilled exterior \(A_{f(a)}\) has a nonempty outer boundary collar on
which \(\Psi_\varepsilon\) is the identity, so \(\Psi_\varepsilon\) is
orientation-preserving.  The boundary action of \(h_\varepsilon\), in the
basis \((\mu_J,\lambda_J)\), has matrix
\[
\begin{pmatrix}
1&0\\
0&\varepsilon
\end{pmatrix}
\]
and determinant \(\varepsilon\).  A diffeomorphism of a manifold with
boundary preserves the inward-pointing side, so its orientation sign is the
same as the sign of its boundary restriction.  Thus \(h_\varepsilon\) has
orientation sign \(\varepsilon\).  The map
\(s\mapsto s^\varepsilon\) also has sign \(\varepsilon\), and therefore
\(\Theta_\varepsilon\) is orientation-preserving on
\(S^1_s\times E(J)\).

Finally,
\[
\Omega_\varepsilon\circ\Phi_f:
(S^4,\Sigma_{g,a,J})\to(S^4,\Sigma_{g,a,J})
\]
is an orientation-preserving self-diffeomorphism of pairs.  Since \(\Phi_f\)
induces \(f\) and \(\Omega_\varepsilon\) induces the identity under the
canonical markings, the composite induces \(f\).  Hence
\[
\Stab_{\Mod(\Sigma_g)}(q_0)
\cap
\Stab_{\Mod(\Sigma_g)}(\Gamma_\mu(J)\cdot[a])
\subseteq
E(\Sigma_{g,a,J}).
\]
\end{proof}

\section{Proofs of the main theorems and a consequence}
\label{sec:proofs-main}

\begin{proof}[Proof of Theorem~\ref{thm:main}]
The upper bound is Proposition~\ref{prop:upper}, and the lower bound is
Proposition~\ref{prop:lower}.
\end{proof}

\begin{corollary}
\label{cor:infinite-index-drop}
Under the hypotheses of Theorem~\ref{thm:main}, the extendable subgroup of
the rim-surgered surface is an infinite-index subgroup of the extendable
subgroup of the standard unknotted surface:
\[
E(\Sigma_{g,a,J})
=
\Stab(q_0)\cap\Stab(\Gamma_\mu(J)\cdot[a])
\subset
\Stab(q_0)
=
E(\Sigma_g^0),
\]
and the inclusion has infinite index.
\end{corollary}

\begin{proof}
By Theorem~\ref{thm:main},
\(
E(\Sigma_{g,a,J})
=
\Stab(q_0)\cap\Stab(\Gamma_\mu(J)\cdot[a]).
\)
By Hirose's theorem, the extendable subgroup of the unknotted surface is
\(
E(\Sigma_g^0)=\Stab(q_0).
\)
It remains to show that
\(
\Stab(q_0)\cap\Stab(\Gamma_\mu(J)\cdot[a])
\)
has infinite index in \(\Stab(q_0)\).

Choose \(b\in H_1(\Sigma_g;\mathbb Z)\) with \(a\cdot b=1\), and let \(T_b\)
be the Dehn twist about a simple closed curve representing \(b\).  For every
\(n\in\mathbb Z\), the even power \(T_b^{2n}\) acts trivially on
\(H_1(\Sigma_g;\mathbb Z_2)\), and therefore preserves the quadratic form
\(q_0\).  Hence
\(
T_b^{2n}\in \Stab(q_0).
\)
On integral homology,
\(
T_b^{2n}([a])=[a]+2n[b].
\)

We claim that the finite subsets
\(
T_b^{2n}\bigl(\Gamma_\mu(J)\cdot[a]\bigr)
\)
are pairwise distinct as \(n\) varies.  Since
\(
\Gamma_\mu(J)\cdot[a]\subseteq\{[a],-[a]\},
\)
it is enough to check that the unordered pairs
\(
\{\pm([a]+2n[b])\}
\)
are pairwise distinct.  If
\[
[a]+2n[b]=[a]+2m[b],
\]
then \(n=m\).  If
\[
[a]+2n[b]=-[a]-2m[b],
\]
then
\[
2[a]+2(n+m)[b]=0,
\]
which is impossible because \(a\cdot b=1\), so \([a]\) and \([b]\) are
linearly independent over \(\mathbb Z\).  Thus the
\(\Stab(q_0)\)-orbit of the finite set \(\Gamma_\mu(J)\cdot[a]\) is infinite.

Therefore the stabilizer
\(
\Stab(q_0)\cap\Stab(\Gamma_\mu(J)\cdot[a])
\)
has infinite index in \(\Stab(q_0)\). 
\end{proof}

\begin{proof}[Proof of Theorem~\ref{thm:intro-pair-classification}]
First prove necessity.  Suppose
\[
F:(S^4,\Sigma_{g,a,J})\to(S^4,\Sigma_{g,b,K})
\]
is an orientation-preserving pair diffeomorphism inducing the prescribed
mapping class \(f\) under the canonical markings.

By Lemma~\ref{lem:rokhlin}, both rim-surgery surfaces have Rokhlin form
\(q_0\) under the canonical markings.  Since \(F\) preserves the induced spin
structure on the surface, we have
\(
f^*q_0=q_0.
\)

The same peripheral-comparison argument used in the proof of
Proposition~\ref{prop:upper}, applied now to a diffeomorphism from the
\((a,J)\)-rim-surgery complement to the \((b,K)\)-rim-surgery complement,
gives the following.  After the basepoint convention, the induced isomorphism
\(
\varphi:G_J\to G_K
\)
satisfies
\(
\varphi(\mu_J)=\mu_K.
\)
By uniqueness of tubular neighborhoods and isotopy extension
\cite[Chapter~4, Theorem~5.3 and Chapter~8, Theorem~1.3]{Hirsch1976}, isotope
\(F\) through pair diffeomorphisms, supported in a tubular neighborhood of
the source surface and fixed on that surface, so that it is an oriented
normal-bundle map near the surface.  This does not change \(f\) or the
induced outer isomorphism of complement groups, and no ambient boundary
condition is present.  The restriction to the normal circle bundle therefore
has the same form used in Proposition~\ref{prop:upper}.

Naturality of the surface-boundary formula gives, for every loop
\(c\subset\Sigma_g\),
\[
\varphi(\lambda_J^{a\cdot c})
=
\lambda_K^{b\cdot f_*(c)}\mu_K^{u([c])}
\]
for some homomorphism
\[
u:H_1(\Sigma_g;\mathbb Z)\to\mathbb Z.
\]
Choosing \(c_0\) with \(a\cdot c_0=1\) shows that
\(\varphi(\lambda_J)\in P_K\).  Applying the same argument to \(F^{-1}\) gives
\(
\varphi(P_J)=P_K.
\)
Hence
\(
\varphi(\lambda_J)=\mu_K^k\lambda_K^\varepsilon
\)
for some \(k\in\mathbb Z\) and \(\varepsilon\in\{\pm1\}\).  Since the
preferred longitudes are null-homologous and
\[
H_1(E(K);\mathbb Z)\cong\mathbb Z\langle\mu_K\rangle,
\]
passing to homology gives \(k=0\).  Thus
\(
\varphi(\lambda_J)=\lambda_K^\varepsilon.
\)
By Theorem~\ref{thm:peripheral-realization}, there is a possibly
orientation-reversing diffeomorphism
\[
h:E(J)\to E(K)
\]
such that
\[
h_*(\mu_J)=\mu_K,
\qquad
h_*(\lambda_J)=\lambda_K^\varepsilon.
\]

Substituting \(\varphi(\lambda_J)=\lambda_K^\varepsilon\) into the naturality
equation and using peripheral injectivity gives
\[
u([c])=0,
\qquad
b\cdot f_*(c)=\varepsilon(a\cdot c)
\]
for every \(c\).  Since \(f_*\) preserves the algebraic intersection form,
\(
b\cdot f_*(c)=f_*^{-1}[b]\cdot[c].
\)
By nondegeneracy of the intersection pairing,
\(
f_*^{-1}[b]=\varepsilon[a],
\)
or equivalently
\(
f_*[a]=\varepsilon[b].
\)
This proves necessity.

Now prove sufficiency.  Suppose \(\varepsilon\in\{\pm1\}\) and a possibly
orientation-reversing diffeomorphism
\[
h:E(J)\to E(K)
\]
satisfy
\[
f^*q_0=q_0,
\qquad
f_*[a]=\varepsilon[b],
\]
and
\[
h_*(\mu_J)=\mu_K,
\qquad
h_*(\lambda_J)=\lambda_K^\varepsilon.
\]

By Hirose's theorem \cite{Hirose2002} and
Lemma~\ref{lem:normal-product}, choose an orientation-preserving ambient
extension of \(f\) which, in the fixed tubular coordinates, is
\[
(p,z)\longmapsto(f(p),z)
\]
near \(\Sigma_g^0\).  Choose the product tubular neighborhood of
\(R_{f(a)}\) to be the image of the chosen product tubular neighborhood of
\(R_a\).  The restriction on the drilled exterior is then literally the
product-coordinate map
\[
(\alpha_a,\beta,\delta_a)
\longmapsto
(\alpha_{f(a)},\beta,\delta_{f(a)}).
\]
Together with the identity on \(S^1_s\times E(J)\), it agrees literally with
the source and target rim-surgery gluings and therefore gives an
orientation-preserving pair diffeomorphism
\[
\Phi_f:
(S^4,\Sigma_{g,a,J})
\longrightarrow
(S^4,\Sigma_{g,f(a),J})
\]
inducing exactly \(f\) under the canonical markings.

Since
\(
[f(a)]=\varepsilon[b],
\)
Proposition~\ref{prop:homology-relative} applies when \(\varepsilon=1\), and
Proposition~\ref{prop:signed-homology-relative} applies when
\(\varepsilon=-1\).  Thus there is a diffeomorphism
\[
\Psi:A_{f(a)}\to A_b
\quad\mathrm{rel}\ \partial M
\]
whose internal-boundary restriction, by the framed tubular-neighborhood
construction used in those propositions, is literally the product-coordinate
map
\[
(\alpha_{f(a)},\beta,\delta_{f(a)})
\longmapsto
(\alpha_b^\varepsilon,\beta,\delta_b^\varepsilon).
\]

Choose on \(\partial E(J)\) the diffeomorphism carrying the parameterized
meridian and longitude circles literally to \(\mu_K\) and
\(\lambda_K^\varepsilon\).  Since a torus diffeomorphism is determined up to
isotopy by its action on \(H_1\) \cite[Theorem~2.5]{FarbMargalit2012}, with
the orientation-reversing case reduced to the orientation-preserving case by
a fixed reflection, the restriction of \(h\) has this mapping class.  An isotopy extended across a boundary collar makes \(h\) equal to this
chosen boundary map and product with it on a smaller collar.  The isotopy does
not change the induced group isomorphism or the orientation sign.  Define
\[
\Theta:S^1_s\times E(J)\to S^1_s\times E(K)
\]
by
\[
\Theta(s,x)=(s^\varepsilon,h(x)).
\]
Its boundary restriction is literally the product of
\[
s\longmapsto s^\varepsilon,
\qquad
\mu_J\longmapsto\mu_K,
\qquad
\lambda_J\longmapsto\lambda_K^\varepsilon.
\]
Under the source gluing, followed by \(\Psi\), and under \(\Theta\), followed
by the target gluing, the common parameterized \(3\)-torus is sent in both
cases by the same product map
\[
s\longmapsto\alpha_b^\varepsilon,
\qquad
\mu_J\longmapsto\beta,
\qquad
\lambda_J\longmapsto\delta_b^\varepsilon.
\]
Thus the restrictions agree literally, and \(\Psi\) and \(\Theta\) glue to a
diffeomorphism of complements.  Since \(\Psi\) is the identity on the outer
boundary collar, the glued map extends over the surface normal disk bundle by
the identity and induces the identity under the canonical markings.

The second-stage map is orientation-preserving.  The map \(\Psi\) is the
identity near the nonempty outer boundary of the connected drilled exterior
and therefore preserves orientation.  The boundary matrix of \(h\) in the
basis \((\mu_J,\lambda_J)\) has determinant \(\varepsilon\), so \(h\) has
orientation sign \(\varepsilon\).  The map \(s\mapsto s^\varepsilon\) has the
same sign, and hence \(\Theta\) is orientation-preserving.  Composing the
second-stage pair diffeomorphism with \(\Phi_f\) gives an
orientation-preserving diffeomorphism
\[
F:(S^4,\Sigma_{g,a,J})
\longrightarrow
(S^4,\Sigma_{g,b,K})
\]
which induces the prescribed mapping class \(f\).  This proves sufficiency.
\end{proof}

\end{document}